\documentclass{amsart}
\usepackage{graphicx}
\usepackage{amssymb}
\usepackage{amsfonts}
\usepackage{hyperref}
\usepackage{mathrsfs}
\swapnumbers
\newtheorem{theorem}{Theorem}[section]
\newtheorem{lemma}[theorem]{Lemma}

\theoremstyle{definition}

\newtheorem{assumption}[theorem]{Assumption}
\newtheorem{remark}[theorem]{Remark}

\numberwithin{equation}{section}
\theoremstyle{plain}

\numberwithin{equation}{section} 
\numberwithin{figure}{section} 
\theoremstyle{plain}
\theoremstyle{plain}
\theoremstyle{remark}
\newtheorem*{acknowledgement*}{Acknowledgement}

\newcommand{\cA}{{\mathcal A}}
\newcommand{\cB}{{\mathcal B}}

\newcommand{\cE}{{\mathcal E}}
\newcommand{\cF}{{\mathcal F}}
\newcommand{\cG}{{\mathcal G}}
\newcommand{\cH}{{\mathcal H}}
\newcommand{\cI}{{\mathcal I}}

\newcommand{\cK}{{\mathcal K}}
\newcommand{\cL}{{\mathcal L}}

\newcommand{\cO}{{\mathcal O}}
\newcommand{\cP}{{\mathcal P}}

\newcommand{\cS}{{\mathcal S}}

\newcommand{\cX}{{\mathcal X}}
\newcommand{\cY}{{\mathcal Y}}

\newcommand{\te}{{\theta}}

\newcommand{\Om}{{\Omega}}
\newcommand{\om}{{\omega}}
\newcommand{\ve}{{\varepsilon}}
\newcommand{\del}{{\delta}}
\newcommand{\Del}{{\Delta}}
\newcommand{\gam}{{\gamma}}
\newcommand{\Gam}{{\Gamma}}

\newcommand{\sig}{{\sigma}}
\newcommand{\al}{{\alpha}}
\newcommand{\be}{{\beta}}

\newcommand{\la}{{\lambda}}

\newcommand{\bbC}{{\mathbb C}}
\newcommand{\bbE}{{\mathbb E}}

\newcommand{\bbN}{{\mathbb N}}

\newcommand{\bbR}{{\mathbb R}}
\newcommand{\bbT}{{\mathbb T}}
\newcommand{\bbZ}{{\mathbb Z}}
\newcommand{\bbI}{{\mathbb I}}

\begin{document}

\title[]{On the Asymptotic moments and Edgeworth expansions for some processes in random dynamical environment}
 \vskip 0.1cm
 \author{Yeor Hafouta \\
\vskip 0.1cm
Department of mathematics\\
The Ohio State University and the Hebrew University}%
\email{yeor.hafouta@mail.huji.ac.il}%

\thanks{ }
\dedicatory{  }
 \date{\today}

\begin{abstract}\noindent
We prove that certain asymptotic moments exist for some random distance expanding dynamical systems
and Markov chains in random dynamical environment, and compute them in terms of the
derivatives at  $0$ of an appropriate pressure function. It will follow  
that these moments satisfy the relations
that the asymptotic moments $\gam_k=\lim_{n\to\infty}n^{-[\frac k2]}\bbE(\sum_{i=1}^n X_i)^k$ of sums of independent and identically distributed centered random variables satisfy. Under certain mixing conditions we will also estimate the convergence rate towards these limits. The arguments in the proof of these results yield that the partial sums generated by the random Ruelle-Perron-Frobenius triplets and all of their parametric derivatives (considered as functions on the base) corresponding to appropriate random transfer or Markov operators satisfy several probabilistic limit theorems such as the central limit theorem.
We will also obtain certain (Edgeworth) asymptotic
expansions related to the central limit theorem for such processes. Our proofs rely on a (parametric) random complex Ruelle-Perron-Frobenius theorem, which replaces some of the spectral techniques used in literature in order to 
obtain limit theorems for deterministic dynamical systems and Markov chains.
\end{abstract}
\maketitle
\markboth{Y. Hafouta}{Asymptotic moments and Edgeworth expansions}
\renewcommand{\theequation}{\arabic{section}.\arabic{equation}}
\pagenumbering{arabic}

 \section{Introduction}\label{sec1}\setcounter{equation}{0}

Probabilistic limit theorems for dynamical systems and Markov chains is a well studied topic. 
One  way to derive such results is relying on some quasi-compactness (or spectral gap)
of an appropriate transfer or Markov operator, together with a suitable perturbation theorem 
(see \cite{Neg1}, \cite{Neg2}, \cite{GH} and \cite{HH}). This  quasi-compactness  can often be verified only via 
an appropriate Ruelle-Perron-Frobenius (RPF) theorem, which is the main key for thermodynamic formalism type constructions.
Probabilistic limit theorems for
random dynamical systems  and Markov chains in random dynamical environments were also studied in the literature (see, for instance, \cite{Kifer-1996}, \cite{Kifer-1998}, \cite{book}, \cite{Aimino} ,\cite{drag} and references therein). In these circumstances, the probabilistic behaviour of the appropriate process is determined by compositions of random operators, and not of a single operator, so no spectral theory can be exploited, and instead, many of these results rely on an appropriate version  
of the RPF theorem for random operators. For instance, the large deviations theorem in \cite{Kifer-1996} rely on 
such results for real operators, and 
the conditions guaranteeing that the central limit theorem in \cite{Kifer-1998}  holds true
 can be often verified by an RPF theorem for random real operators. Relying on certain contraction properties of random complex transfer and Markov operators, with respect to a complex version of the Hilbert projective metric due to H.H. Rugh \cite{Rug} (see also \cite{Dub1} and \cite{Dub2}), we proved in \cite{book} an RPF theorem for random complex operators and presented the appropriate random complex thermodynamic formalism type constructions,
which was one of the main keys in the proof of  versions of the Berry-Esseen theorem and the local central limit 
theorem (LCLT) for certain processes in random dynamical environment, and in the proof of some nonconventional LCLT for dynamical systems (see Chapters 2 and 7 of \cite{book}).
 In this paper we  use this RPF theorem in order to derive additional limit theorems, as described in the following paragraphs.

Let $(\Om,\cF,P,\te)$ be an invertible measure preserving system, let $\cX$ be a compact metric space, $\cE_\om\subset\cX$ be a measurable family of compact subsets and $u_\om:\cE_\om\to\bbR$ be a random function, where $\om\in\Om$. Let $T_\om:\cE_\om\to\cE_{\te\om}$ 
be a random distance expanding map  and let
$\xi_0^{\om},\xi_1^{\te\om},\xi_2^{\te^2\om},...$ be a Markov chain in the random dynamical environment $(\Om,\cF,P,\te)$. In this paper we consider sequences of random variables having either the form  $S_n^\om=\sum_{j=0}^{n-1}u_{\te^j\om}\circ T_\om^j(x)$ or the form $\cS_n^\om =\sum_{j=0}^{n-1}u_{\te^j\om}(\xi_j^{\te^j\om})$, where 
$T_\om^j=T_{\te^{j-1}\om}\circ\cdots\circ T_{\te\om}\circ T_\om$ and $x$ is distributed according to a special (Gibbs) measure $\mu_\om$. 
Recall the following result. Let $X_1,X_2,X_3,...$ be a sequence of centered, independent and identically distributed random variables with finite moments of all orders, and set $\sig^2=\bbE X_1^2$ and $\zeta=\bbE X_1^3$.
Then  for each integer  $k\geq 2$ the limit 
\[
\gam_k=\lim_{n\to\infty}n^{-[\frac k2]}\bbE\big(\sum_{j=1}^n X_j\big)^k
\]
of the $k$-th normalized moment
exists. Moreover, $\gam_2=\sig^2$ and $\gam_3=\zeta$, and when
 $k$ is even we have $\gam_k=C_k\sig^k$, where $C_k=2^{-\frac k2}(\frac{k}{2}!)^{-1}k!$, while for odd $k$'s larger than $2$ we have 
$\gam_k=D_k\sig^{k-3}\zeta$, where $D_k=\frac{k!}{3!}2^{-\frac12(k-3)}(\frac{k-3}{2}!)^{-1}$. This result is a direct consequence of the multinomial theorem, and we refer the readers to Lemma 4.1 and Corollary 4.2 in \cite{Liverani} for a generalization to several weakly dependent sequences of random variables.
In this paper, we will prove generalize these results for the sequences  $S_n^\om$ and $\cS_n^\om$ defined above by showing that the $k$-th normalized moments of the random variables $S_n^\om$ and $\cS_n^\om$
behave  like powers of ergodic averages of the second and third derivative at $0$ of some pressure function, which will imply that the asymptotic moments exist, satisfy the desired relations and that they do not depend on $\om$ when $(\Om,\cF,P,\te)$ is ergodic. 
The fact that the second (normalzied) moments behave like an ergodic average was proved in \cite{Kifer-1998} in a more general setup, and here we show that this is true (in our context) also for all the higher moments. 

When $(\Om,\cF,P,\te)$ satisfies some mixing conditions then we obtain almost sure convergence rate of order $n^{-\frac12}\ln n$ towards the asymptotic moments, which allows us to obtain almost optimal convergence rate (of the same order) in the corresponding quenched central limit theorem for the sequences $n^{-\frac12}S_n^\om$ and $n^{-\frac12}\cS_n^\om$. The arguments in the proof  of the above convergence rate also yields that for sufficiently well mixing base maps $\te$ the partial sums generated by the random RPF triplets (corresponding to an appropriate parametric family of random transfer operators) and  all of their parametric derivatives satisfy several probabilistic limit theorems (see Remark \ref{RPF trip LimThms}), which also seems to be a new result.

Let $X_1,X_2,X_3,...$ be a sequence of centered  iid random variables with finite moments. Then,
under certain assumptions on the behavior of the characteristic function of $X_1$, a classical result (see \cite{Feller}) states that (Edgeworth) expansions of the form
\[
\sup_{s\in\bbR}\big|\sqrt{2\pi}P(\sum_{j=1}^n X_j\leq\sig\sqrt n s)-\int_{-\infty}^s e^{-\frac{t^2}{2}}dt-\sum_{j=1}^d n^{-\frac{j}2}P_j(s)e^{-\frac{s^2}2}\big|=o(n^{-\frac{d}2})
\]
hold true for some polynomials $P_1,P_2,...,P_d$. 
Edgeworth expansion have been obtained  in \cite{Neg2} for some classes of Markov chains and in \cite{Coelho} and \cite{Liverani}  for deterministic dynamical systems (and Markov chains), and in this paper 
we will also obtain Edgeworth type expansions for the  processes $S_n^\om$ and $\cS_n^\om$  described earlier,
with random polynomials $P_i=P_{i,\om,n}$ whose degree depends only on $i$, and the coefficients of $P_{i,\te^{-n}\om,n}$
converge as $n\to\infty$. 
When the dynamical environment $(\Om,\cF,P,\te)$ satisfies certain mixing conditions we also obtain certain convergence rate in the later limits. In general, in order to obtain Edgeworth type expansions by analysing the asymptotic behaviour of the characteristic functions 
$\varphi_n(t)=\bbE e^{it S_n}$ of an underlying sequence $S_n,\,n\geq1$ of random variables, it is also necessary to have estimates on the decay of these functions (as $n\to\infty$) for large $t$'s, which is beyond the applications of the random complex 
thermodynamic formalism (the decay estimates sufficient for the
existence of Edgeworth expansions in the deterministic case are discussed
in \cite{Coelho} and \cite{Liverani}). Here $\bbE$ denotes expectation with respect to $P$. 
In the deterministic case, some of these estimates were obtained using the spectral theory of appropriate perturbations of the original transfer or Markov operator. In the circumstances of this paper there is no single operator, but a random family of operators, and we will use the  approach taken in \cite{book} in order to obtain the appropriate estimates in the case of iterates of random operators.

The paper is organized as follows. In Section \ref{sec2} we  formulate our results for random variables $S_n^\om$
whose characteristic function is related to certain random operators $\cA_{it}^\om,\,t\in\bbR$ which  satisfy some complex thermodynamic formalism type assumption, together with certain decay rates of the norms of the iterates 
$\cA_{it}^{\om,n}=\cA_{it}^{\te^{n-1}\om}\circ\cA_{it}^{\te^{n-2}\om}\circ\cdots\circ\cA_{it}^\om$ of these operators for large $t$'s.
In Section \ref{MomSec} we  prove the results concerning the normalized asymptotic moments and in Section \ref{EdgeSec} we  prove the results concerning Edgeworth expansions. Section \ref{secEXM} is devoted to examples of random dynamical systems $T_\om$ and Markov chains $\xi_0^\om,\xi_1^{\te\om},\xi_2^{\te^2\om},...$ in random environments satisfying our assumptions.

\section{Preliminaries and main results}\setcounter{equation}{0}\label{sec2}
Our setup consists of a complete probability space 
$(\Om,\cF,P)$ together with an invertible $P$-preserving transformation $\te:\Om\to\Om$,
of a compact metric space $(\cX,\rho)$ normalized in size so that 
$\text{diam}\cX\leq 1$ together with the Borel $\sig$-algebra $\cB$, and
of a set $\cE\subset\Om\times \cX$ measurable with respect to the product $\sig$-algebra $\cF\times\cB$ such that the fibers $\cE_\om=\{x\in \cX:\,(\om,x)\in\cE\},\,\om\in\Om$ are compact. The latter yields (see \cite{CV} Chapter III) that the mapping $\om\to\cE_\om$ is measurable with respect to the Borel $\sig$-algebra induced by the Hausdorff topology on the space $\cK(\cX)$ of compact subspaces of $\cX$ and the distance function $\rho(x,\cE_\om)$ is measurable in $\om$ for each $x\in \cX$.  
Furthermore, the projection map $\pi_\Om(\om,x)=\om$ is
measurable and it maps any $\cF\times\cB$-measurable set to an
$\cF$-measurable set (see ``measurable projection" Theorem III.23 in \cite{CV}).

Let $(H_1^\om,\|\cdot\|_1)$ be a Banach space of complex valued functions on $\cE_\om$, containing the constant functions, so that $\|g\|_1\geq\sup|g|$ for any $g\in H_1^\om$, and $\|\textbf{1}\|_1=1$,
where $\textbf{1}$ denotes the function taking the constant value $1$ (regardless of its domain).
Denote by $(H_1^\om)^*$ the space of all continuous linear functionals on $(H_1^\om,\|\cdot\|_1)$, equipped with the operator norm $\|\cdot\|_1$.
Let $\cA_z^\om:H_1^\om\to H_1^{\te\om},\,z\in\bbC$ be a family of continuous linear operators, $S_n^\om:\cE_\om\to\bbR,\,n\geq1$ be a sequence of Borel measurable 
functions and $\mu_\om$ be a random probability measure on $\cE_\om$
so that $P$-a.s. we have $\mu_\om(S_n^\om)=0$, $\|S_n^\om\|_{L^k(\cE_\om,\mu_\om)}<\infty$ for any $k\geq1$ 
and  for any $z\in\bbC$ and $n\geq1$,
\[
\mu_\om(e^{zS_n^\om})=\mu_{\te^n\om}(\cA_z^{\om,n}\textbf{1})
\]
where $\cA_z^{\om,n}=\cA_z^{\te^{n-1}\om}\circ\cA_z^{\te^{n-2}\om}\circ\cdots\circ\cA_z^\om$.
 In this paper we will study certain (asymptotic) properties the distribution of 
$S_n^\om(x)$, when $x$ is distributed according to $\mu_\om$,  and $\om$ ranges over a set of full $P$-probability.
Our basic requirements from the operators are described in the following
\begin{assumption}\label{BaseAss}
There exist constants $\rho,C>0$ and $\del\in(0,1)$  so that $P$-a.s. for any 
$z\in B(0,\rho):=\{\zeta\in\bbC: |\zeta|<\rho\}$ 
there is a 
triplet  consisting of a nonzero complex number $\la_\om(z)$,  a function $h_\om(z)\in H_1^\om$
and a linear functional $\nu_\om(z)\in(H_1^\om)^*$ such that $\nu_\om(z)\textbf{1}=1$ and for any
 $g\in H_1^\om$ and $n\geq1$,
\begin{equation}\label{ExpConvAss}
\left\|(\la_{\om,n}(z))^{-1}\cA_z^{\om,n}g-\nu_\om(z)(g)h_{\te^n\om}(z)\right\|_1\leq C\|g\|_1\del^n
\end{equation}
where  $\la_{\om,n}(z)=\prod_{j=0}^{n-1}\la_{\te^j\om}(z)$.
Moreover, the above triplet is measurable in $\om$, analytic in $z$ and the random variables $\sup_{z: |z|<\rho}|\la_\om(z)|$,  $\sup_{z: |z|<\rho}\|h_\om(z)\|_1$ and $\sup_{z: |z|<\rho}\|\nu_\om(z)\|_1$ are bounded. Furthermore, 
$\nu_\om(0)=\mu_\om$, $h_\om(0)=\textbf{1}$ and $\la_\om(0)=1$.
\end{assumption}
We refer the readers' to Section \ref{secEXM} for several examples in which Assumption \ref{BaseAss} holds true.
Under Assumption \ref{BaseAss} we have $\nu_\om(z)(h_\om(z))=1$, $\cA_z^\om h_\om(z)=\la_\om(z)h_{\te\om}(z)$
and $(\cA_z^\om)^*\nu_{\te\om}(z)=\la_\om(z)\nu_\om(z)$ (see the beginning of Section \ref{MomSec}), 
where $(\cA_z^\om)^*: (H_1^{\te\om})^*\to (H_1^\om)^*$ is the dual operator of $\cA_z^\om$ (i.e. the triplet from the assumption is a random complex Ruelle-Perron-Frobenius triplet, see Chapter 4 in \cite{book}). 
In particular $\cA_0^\om \textbf{1}=\textbf{1}$.


Our first result is the following

\begin{theorem}\label{MomThm}

(i) By possibly decreasing $\rho$, $P$-a.s. we can define an analytic function $\Pi_\om:B(0,\rho)\to\bbC$
so that $\Pi_\om(0)=0$, $\la_\om(z)=e^{\Pi_\om(z)}$ and 
$|\Pi_\om(z)|\leq c_0$ for any $z\in B(0,\rho)$, for some constant $c_0$ which does not depend on $\om$
and $z$.

(ii) Suppose that $\int S_n^\om\mu_\om=0$ for $P$-almost any $\om$. For any $k\geq 2$ and $n\geq1$, let $\gam_{k,n}$ be the random variable defined by 
$\gam_{k,n}(\om)=n^{-[\frac k2]}\int_{\cE_\om} (S_n^\om(x))^kd\mu_\om(x)$. Then, the limits
\[
\gam_k(\om)=\lim_{n\to\infty}\gam_{k,n}(\om)=\lim_{n\to\infty}\gam_{k,n}(\te^{-n}\om)
\]
exist $P$-a.s. and in $L^p$, for any $1\leq p<\infty$. Moreover,
with $\sig^2=\gam_2$, $\zeta=\gam_3$, $C_k=2^{-\frac k2}(\frac{k}{2}!)^{-1}k!$ and 
$D_k=\frac{k!}{3!}2^{-\frac12(k-3)}(\frac{k-3}{2}!)^{-1}$,
when $k$ is even we have $\gam_k=C_k\sig^k$, while for odd $k$'s we have 
$\gam_k=D_k\sig^{k-3}\zeta$. Furthermore,
$\sig^2=\gam_2=\bbE[\Pi_\om''(0)|\cI]$ and 
$\zeta=\gam_3=\bbE[\Pi_\om'''(0)|\cI]$, where $\cI$ is the sub-$\sigma$ algebra of $\cF$ containing only the $\te$-invariant sets.
In particular, 
when the measure preserving system $(\Om,\cF,P,\te)$ is ergodic then all $\gam_k(\om)$'s do not depend on $\om$.

(iii) In fact, we can write 
\[
\int_{\cE_\om} (S_n^\om(x))^kd\mu_\om(x)=\mu_{\te^n\om}(h_{\te^n\om}^{(j)}(0))+\sum_{s=1}^{[\frac{k}{2}]}C_{\om,n}^{(k)}(s)n^s+\ve_{\om,n}
\]
where $|\ve_{\om,n}|\leq c_1\eta^n$ for some constants $c_1>0$ and $0<\eta<1$, 
\[
\lim_{n\to\infty}C_{\om,n}^{(k)}([\frac{k}{2}])=\lim_{n\to\infty}C_{\te^{-n}\om,n}^{(k)}([\frac{k}{2}])=\gam_k(\om)
\]
 $C_{\om,n}^{(k)}([\frac{k}{2}]-1)=C_{\om,n}^{(k)}(1)=0$ and for all other $s$'s, 
\begin{eqnarray*}
\lim_{n\to\infty}C_{\te^{-n}\om,n}^{(k)}(s)=\\
\sum_{j=0}^k j!\binom{k}{j}\mu_\om(h_\om^{(k-j)}(0))\sum_{m_2,...,m_j}\big(\prod_{l=2}^j(i!)^{m_l}m_l!\big)^{-1}\prod_{l=2}^j(\bbE_P[\Pi_\om^{(l)}|\cI])^{m_l}.
\end{eqnarray*}
Here the sum ranges over all possible choices of nonnegative integers $m_2,...,m_j$ so that $\sum_{l=2}^j m_l=s$
and $\sum_{l=2}^j lm_l=j$, and $f^{(l)}$ stands for the $l$-th derivatives of a function $f$ on $B(0,\rho)$ which takes values at some Banach space. Moreover, closed formulas for the $C_{\om,n}^{(k)}$'s can be recovered from the proof.
\end{theorem}
As in \cite{Liverani}, the method used in the proof of Theorem \ref{MomThm} yields that the limits $\gam_k$ exist also without the assumption that $\int S_n^\om d\mu_\om=0$, but then the relations between the $\gam_k$'s are more complicated and will not be discussed in this paper.

Now we will describe our results concerning  Edgeworth type expansions.
For any $r>0$ consider the following assumptions

\begin{assumption}\label{A3}
(i) There exists a random variable $a_\om$ so that  $P$-a.s. we have 
$\|\cA_{it}^{\om,n}\|_1\leq a_\om$ for any $t\in(-\rho,\rho)$, where $\rho$ comes from Assumption \ref{BaseAss}.

(ii) For any compact set $J\subset\bbR\setminus\{0\}$, $P$-.a.s for any $n\geq1$ we have 
\[
\lim_{n\to\infty}n^{r+\frac12}\sup_{t\in J}\|\cA^{\om,n}_{it}\|_1=0.
\]
\end{assumption}

\begin{assumption}\label{A4}
There exists a Banach space $(H_2^\om, \|\cdot\|_2)$ containing the constant functions, such that $\|g\|_2\geq\sup|g|$ for any $g\in H_2^\om$, and $\cA_{it}^\om$, $t\in\bbR$, can be (possibly)
extended to $H_2^\om$ so that $\cA_{it}$ is a continuous map between $H_2^\om$ and $H_2^{\te\om}$. Moreover:

(i) There exists a random variable $R_\om$ so that $P$-a.s.
 for any $t\in\bbR$ and $n\geq1$ we have $\|\cA_{it}^{\om,n}\|_2\leq R_\om(1+|t|)$.

(ii) There exist random variables $K_\om,D_\om,C_\om>0$, $r_2(\om)>0$  and $N(\om)$ so that $P$-a.s. for any  $n\geq N(\om)$,
\begin{equation}\label{EqA4}
\sup_{t: K_\om\leq |t|\leq D_\om n^{r}}\|\cA_{it}^{\om,n}\|_2\leq C_\om n^{-r_2(\om)}.
\end{equation}
\end{assumption}

\begin{theorem}\label{EdgeThm}
Suppose that Assumptions \ref{BaseAss} and \ref{A3} hold true with $r=\frac12$. Set 
$\Pi_{\om,n,2}=n^{-1}\sum_{j=0}^{n-1}\Pi_{\te^j\om}''(0)$, and assume that $\sig^2=\lim_{n\to\infty}\Pi_{\om,n,2}=\lim_{n\to\infty}\Pi_{\te^{-n}\om,n,2}=\bbE[\Pi_\om''(0)]|\cI]>0$.

(i) There exists a sequence of polynomials
$P_{\om,n,1}(s)=\sum_{j=0}^{m_1}a_{\om,n,j,1}s^j,\,n\geq1$ with random coefficients, whose degree $m_1$ does not depend on $\om$ and $n$, so that $P$-almost surely we have
\begin{eqnarray*}
\lim_{n\to\infty}n^{\frac12}\sup_{s\in\bbR}\Big|\sqrt{2\pi}\mu_\om\{x\in\cE_\om: S_n^\om(x)\leq\sqrt n s\}-\\\frac{1}{\sqrt{\Pi_{\om,n,2}}}\int_{-\infty}^s e^{-\frac{t^2}{2\Pi_{\om,n,2}}}dt- n^{-\frac{1}2}P_{\om,n,1}(s)e^{-\frac{s^2}2}\Big|=0.
\end{eqnarray*}

(ii) Let $d\in\bbN$. If Assumptions \ref{A3} and  \ref{A4} hold true with $r=\frac{d}2$ and $(\Om,\cF,P,\te)$ is ergodic, then for each $k\geq2$ 
there exists a sequences of polynomials 
$P_{\om,n,k}(s)=\sum_{j=0}^{m_k}a_{\om,n,j,k}s^j,\,n\geq1$ with random coefficients and degree $m_k$ which depends only on $k$, so that $P$-almost surely we have
\begin{eqnarray*}
\lim_{n\to\infty}n^{\frac d2}\sup_{s\in\bbR}\Big|\sqrt{2\pi}\mu_\om\{x\in\cE_\om: S_n^\om(x)\leq\sqrt n s\}-\\\frac{1}{\sqrt{\Pi_{\om,n,2}}}\int_{-\infty}^s e^{-\frac{t^2}{2\Pi_{\om,n,2}}}dt-\sum_{j=1}^d n^{-\frac{j}2}P_{\om,n,j}(s)e^{-\frac{s^2}2}\Big|=0
\end{eqnarray*}

(iii) The coefficients $a_{\om,n,j,k}$ of the above
 polynomials are algebraic combinations of the derivatives of the functions $\mu_{\te^n\om}(h_\om(z))$ and $n^{-1}\sum_{j=0}^{n-1}\Pi_{\te^j\om}(z)$ at $z=0$. In particular, they 
are uniformly bounded in $\om$ and $n$, where $\om$ ranges over a set of probability $1$,  and  
$a_{\te^{-n}\om,n,j,k}$ converges $P$-a.s. and in $L^p$, for any $p\in[1,\infty)$, as $n\to\infty$ towards a limit
$a_{\om,j,k}$ 
 The coefficients $a_{\om,n,j,k}$ and the 
corresponding limits can be recovered from the proof. 
\end{theorem}
Note that when (\ref{MomRateTemp}) below holds true  then  we obtain a converge rate of 
order $n^{-\frac12}\ln n$ for all of the coefficients of 
the $P_{\te^{-n}\om,n,k}$'s. 

\begin{remark}
A natural question is whether Theorem \ref{EdgeThm} (or a version of it with $S_n^{\te^{-n}\om}$ instead of $S_n^\om$) can be obtained with polynomials which depend only on $\om$ and not on $n$ and with $\sig^2$ in place of $\Pi_{\om,n,2}$. Similar arguments to the ones in Remark 1.1 from \cite{Liverani} show that replacing $P_{\om,n,k}$ with a polynomial  $P_{\om,k}$ which does not depend on $k$ is possible only if $n^{\frac k2}|P_{\om,k}(x)-P_{\te,n,k}(x)|$ converges to $0$ as $n\to\infty$ uniformly in $x$, which means that the only candidate for such polynomials are the polynomials $P_{\om,k}$ whose coefficients are the limits of the coefficients of the $P_{\om,n,k}$'s (or the $P_{\te^{-n}\om,n,k}$'s in the ``reverse version").
Proving such a result does not seem possible using only the converge of the coefficients of $P_{\te^{-n}\om,n,k}$ even when $\te$ satisfies some mixing conditions, since this convergence  is derived from the convergence of certain ergodic averages, which, in general, is slower than $n^{-\frac12}$ (even for independent summands). Note that even for Edgeworth expansions of order one (as well as for a Berry-Esseen theorem) it is not clear how to obtain the results with $\sig^2$ in place of $\Pi_{\om,n,2}$, since this essentially requires better convergence rate of the variance of $n^{-\frac12}S_n^\om$ towards $\sig^2$. The difference between this variances and $\sig^2$ behaves as an ergodic average, which leads again to the same problem discussed above (concerning the rate of converges of ergodic averages). 
\end{remark}

 

\subsection{Almost sure convergence rates under mixing conditions}\label{MixSec}
When  $(\Om,\cF,P,\te)$ satisfies certain  mixing assumptions, then, in Section \ref{MomRateSec} 
we obtain  almost sure converges rate of the form
\begin{equation}\label{MomRateTemp}
\left|n^{-[\frac k2]}\int_{\cE_\om} (S_n^\om(x))^kd\mu_\om(x)-\gam_k\right|\leq R_{\om,k} n^{-\frac12}\ln n
\end{equation}
where $R_{\om,k}$ is some random variable.
Under certain mixing condition the authors of \cite{Ste} obtained  converges rates of order 
$(\ln n)^{\frac 32+\del_0}n^{-\frac12},\,\del_0>0$ for $k=2$, and in our setup we obtain such rates for all $k$'s, with 
$\del_0=-\frac12$.
 Using (\ref{MomRateTemp}) with $k=2$, when $\sig^2>0$ we 
also derive in Section \ref{MomRateSec} almost optimal convergence rate in the central limit theorem of the form
\[
\sup_{s\in\bbR}\left|\mu_\om\{x\in\cE_\om: S_n^\om(x)\leq s\sqrt n\}-\frac1{\sqrt{2\pi\sig^2}}\int_{-\infty}^s e^{-\frac{t^2}{2\sig^2}}dt\right|\leq 
c_\om n^{-\frac12}\ln n.
\]

\section{Asymptotic moments: Proof of Theorem \ref{MomThm}}\label{MomSec}\setcounter{equation}{0}
First, since $\la_\om(0)=1$, $\la_\om(\cdot)$ is an analytic function and $|\la_\om(z)|$ is bounded 
uniformly in $\om$ and $z$, where $\om$ ranges over a set of probability $1$ and $z\in B(0,\rho)$, it is indeed possible to construct a function $\Pi_\om(\cdot)$ which satisfies the
conditions stated in Theorem \ref{MomThm} (i), namely, an appropriate branch of the logarithm of $\la_\om(z)$ can be defined on some deterministic neighborhood of $0$.

Next, we will show that  $(\la_\om(z),h_\om(z),\nu_\om(z))$ is an RPF triplet, namely that $P$-a.s. for any $z\in B(0,\rho)$,
\begin{equation}\label{RPF}
\nu_\om(z)(h_\om(z))=1,\,\cA_z^\om h_\om(z)=\la_\om(z)h_{\te\om}(z)\,\,\text{ and }\,\,
(\cA_z^\om)^*\nu_{\te\om}(z)=\la_\om(z)\nu_\om(z).
\end{equation}
Indeed,
in order to prove the second equality in (\ref{RPF}), 
plug in $\te^{-n}\om$ in place of $\om$ in (\ref{ExpConvAss}), and then choose $g=\textbf{1}$, in order to deduce that
\begin{eqnarray*}
\cA_z^\om h_\om(z)=\lim_{n\to\infty}(\la_{\te^{-n}\om,n}(z))^{-1}\cA_z^\om(\cA_z^{\te^{-n}\om,n}\textbf{1})\\=
\la_\om(z)\lim_{n\to\infty}(\la_{\te^{-(n+1)}\te\om,n+1}(z))^{-1}
\cA_z^{\te^{-(n+1)}\te\om,n+1}=\la_\om(z)h_{\te\om}(z).
\end{eqnarray*}
Next, applying  (\ref{ExpConvAss}) with $\te\om$ in place of $\om$ and with $\cA_z^\om g$ in place of $g$ we derive that for any $g\in H_1^\om$,
\begin{eqnarray*}
\nu_{\te\om}(z)(\cA_z^\om g)=\lim_{n\to\infty}\frac{\cA_z^{\te\om,n}(\cA_z^\om g)}{\la_{\te\om,n}(z)h_{\te^{n+1}\om}(z)}\\=\la_\om(z)\lim_{n\to\infty}\frac{\cA_z^{\om,n+1}g}{\la_{\om,n+1}(z)h_{\te^{n+1}\om}(z)}=
\la_\om(z)\nu_\om(z)(g)
\end{eqnarray*}
which implies that $(\cA_z^\om)^*\nu_{\te\om}(z)=\la_\om(z)\nu_\om(z)$. Note that the function $h_\om(z)$ does not vanishes 
when $z$ lies in a deterministic neighborhood of $0$ since $h_\om(0)\equiv1$, the random variable $\sup_{|z|<\rho}\|h_\om(z)\|_1$ is bounded and the map $z\to h_\om(z)$ is analytic.   
Finally, since 
\[
h_\om(z)=\lim_{n\to\infty}\frac{\cA_z^{\te^{-n}\om,n}\textbf{1}}{\la_{\te^{-n}\om,n}(z)}
\] 
and 
\[
\la_{\te^{-n}\om,n}(z)=
\la_{\te^{-n}\om,n}(z)\nu_{\te^{-n}\om}(z)\textbf{1}=
\nu_\om(z)(\cA_z^{\te^{-n}\om,n}\textbf{1})
\]
we obtain that $\nu_\om(z)(h_\om(z))=1$. \qed

Now we will start proving Theorem \ref{MomThm}.
First, it follows from (\ref{RPF}) that $\la_{\om,n}(z)=\nu_{\te^n\om}(z)(\cA_z^{\om,n}\textbf{1})$, and 
so, since $\nu_\om(z)\textbf{1}=1$ and $\mu_\om(e^{zS_n^\om})=\mu_{\te^n\om}(\cA_z^{\om,n}\textbf{1})$,
\[
\la_{\om,n}'(0)=\nu_{\te^n\om}'(0)\textbf{1}+
\frac{ d\nu_{\te^n\om}(0)(\cA_z^{\om,n}\textbf{1})}{dz}\Big|_{z=0}=\frac{d\mu_\om(e^{zS_n^\om})}{dz}\Big|_{z=0}=\mu_\om(S_n^\om)=0.
\]
Next, using Assumption \ref{BaseAss} for any $n\geq1$ we can write
\begin{equation}\label{I}
\mu_\om(e^{zS_n^\om })=\mu_{\te^n\om}(\cA_{z}^{\om,n}\textbf{1})=
\la_{\om,n}(z)\big(\mu_{\te^n\om}(h_{\te^n\om}(z))+\del_{\om,n}(z)\big)
\end{equation}
where $|\del_{\om,n}(z)|\leq C\del^n$. Since $\del_{\om,n}(\cdot)$ is an analytic function, it follows from the Cauchy integral formula that
for any $k\geq 1$ there exists a constant $Q_k$ so that 
$|\del_{\om,n}^{(k)}(z)|\leq Q_k\del^n$ for any $z\in B(0,\frac \rho 2)$. Since $\la_{\om,n}(0)=1$ and $\sup_{z\in B(0,\rho)}|\la_\om(z)|$ is a bounded random variable, the analyticity of $\la_\om(\cdot)$ implies that for any $k\geq1$ we have 
$|\la_{\om,n}^{(k)}(0)|\leq H_k n^k$, for some constant $H_k$ which does not depend on $n$ and $\om$. Therefore, differentiating $k$ times both sides of (\ref{I})  at $z=0$ yields that
\begin{equation}\label{II}
\mu_\om(S_n^\om )^k=\sum_{j=0}^k\binom{k}{j}\la_{\om,n}^{(j)}(0)\mu_{\te^n\om}(h_{\te^n\om}^{(k-j)}(0))+d_{\om,n}
\end{equation}
for some random variable $d_{\om,n}$ so that
$|d_{\om,n}|\leq Ac^n$ for some $A\geq1$ and $c\in(0,1)$, which do not depend on $\om$.
Since $\la_{\om,n}(z)=e^{\sum_{j=0}^{n-1}\Pi_{\te^j\om}(z)}$ and $\Pi_\om(0)=1$, it follows from the Fa\'a di Bruno formula that for any $j$,
\begin{equation}\label{***}
\la_{\om,n}^{(j)}(0)=j!\sum_{s=1}^{[\frac j2]}n^s\sum_{(m_2,...,m_k)\in \Gam_{j,s}}
\big(\prod_{l=2}^j(l!)^{m_l}m_l!\big)^{-1}\prod_{l=2}^{j}\Big(n^{-1}\sum_{i=0}^{n-1}\Pi_{\te^i\om}^{(l)}(0)\Big)^{m_l}
\end{equation}
where $\Gam_{j,s}$ is the set of all $j-1$-tuples $(m_2,...,m_j)$ of nonnegative integers 
so that $\sum_l lm_l=j$ and $\sum_{l}m_l=s$,
and we took into account that 
\[
\sum_{j=0}^{n-1}\Pi_{\te^j\om}'(0)=\la_{\om,n}'(0)=\mu_{\om}(S_n^\om)=0.
\] 
Observe that when $s=[\frac j2]$ 
and $j$ is even  we have $\Gam_{j,s}=\{(\frac j2,0,0,...,0)\}$, while for odd $j$'s we have  $\Gam_{j,s}=\{(\frac{j-3}2,1,0,...,0)\}$. Since $\Pi_\om(z)$ is analytic in $z$ and uniformly bounded in $\om$ and $z$
(where $\om$ ranges over a set of probability $1$ and $z\in B(0,\rho)$),
for each $j$ there exists a constant $L_j$ so that 
$\sup_{z\in B(0,\frac \rho 2)}|\Pi_\om^{(j)}(z)|\leq L_j$ for $P$-almost any $\om$.
When $k$ is even,
by considering the case when $j=k$, we conclude that
 there exist constants $R_k, k\geq2$ so that for even $k$'s we have
\begin{equation}\label{Rate even}
\left|\frac{\mu_\om(S_n^\om )^k}{n^{[\frac k2]}}-C_k\Big(n^{-1}\sum_{i=0}^{n-1}\Pi_{\te^i\om}^{(2)}(0)\Big)^{\frac k2}\right|\leq R_kn^{-1}
\end{equation}
where $C_k=2^{-\frac k2}(\frac k2!)^{-1}k!$. Now, suppose that $k$ is odd. Then the dominating terms come from considering $j=k$ and $j=k-1$. Differentiating both sides of the equality $\nu_\om(z)(h_\om(z))=1$ and plugging in $z=0$ we derive that
\[
\mu_\om(h_\om'(0))=-\nu_\om'(0)\textbf{1}=0
\] 
where we used that $h_\om(0)\equiv\textbf{1}$ and that
$\nu_\om(z)\textbf{1}=1$ for any $z$. Therefore, we actually only need to consider the case when $j=k$ also for odd $k$'s, and so, with $D_k=\frac{k!}{3!}2^{-\frac12(k-3)}(\frac{k-3}{2}!)^{-1}$
 for odd $k$'s large than $2$ we have
\begin{equation}\label{Rate odd}
\left|\frac{\mu_\om(S_n^\om )^k}{n^{[\frac k2]}}-D_k\Big(n^{-1}\sum_{i=0}^{n-1}\Pi_{\te^i\om}^{(2)}(0)\Big)^{\frac{k-3}2}\cdot
\Big(n^{-1}\sum_{i=0}^{n-1}\Pi_{\te^i\om}^{(3)}(0)\Big)\right|\leq R_kn^{-1}.
\end{equation}
Since  the $L^\infty$-norms of
$n^{-1}\sum_{i=0}^{n-1}\Pi_{\te^{i}\om}^{(l)}(0)$ are uniformly bounded in $n$ (for each $l$), 
all the positive integer powers of these averages
converge almost surely and in $L^p$, for any $p\in [1,\infty)$, towards the appropriate power of
$E[\Pi_\om^{(l)}(0)|\cI]$. 
Combining this with the above estimates we complete the proof of Theorem \ref{MomThm}.


\subsection{Convergence rates towards the asymptotic moments under mixing conditions}\label{MomRateSec}

We assume here that $\Om=\cY^z$ for some measurable space $\cY$, and that $(\Om,\cF,P,\te)$ is the shift system
generated by a $\cY$-valued stationary process $\{\xi_n:\,n\in\bbZ\}$. Consider the situation when 
$\cA_z^\om$ depends only on 
the $0$-th coordinate of $\om$. Then  by (\ref{ExpConvAss}), when $z$ lies in some deterministic neighborhood of the origin, the functional
$\nu_\om(z)$ can be 
approximated in $L^\infty$ exponentially fast in $n$ by functions of the first $n$-th coordinates. Indeed, since $h_\om(0)=\textbf{1} $, $h_\om(z)$ is analytic and $\sup_{z\in B(0,\rho)}\|h_\om(z)\|_1$ is bounded, we derive that there exist constants $r_0>0$ and $\del_0>0$ so that $P$-a.s. for any $z\in B(0,r_0)$ we have $\inf|h_\om(z)|>\del_0$. 
Since  $\nu_\om(z)\textbf{1}=1$, it follows  now from (\ref{ExpConvAss}) that for any $z\in B(0,r_0)$ and $g\in H_1^\om$ so that $\|g\|_1\leq1$,
\[
\Big\|\nu_{\om}(z)(g)-\frac{\cA^{\om,n}_zg}{\cA^{\om,n}_z \textbf{1}}\Big\|_1\leq A\del^n
\]
where we used that $\nu_{\om}(z)(g)=\frac{\nu_{\om}(z)(g)}{\nu_{\om}(z)\textbf{1}}$. Here $A>0$ is some constant and $\del$ comes from (\ref{ExpConvAss}). 
By (\ref{RPF}) we have $\la_\om(z)=\nu_\om(z)(\cA_z^\om \textbf{1})$, and therefore $P$-a.s. for any $z\in B(0,r_0)$
and $n\geq 1$ we have 
\[
|\la_\om(z)-q_{n,z}(\om_0,\om_1,...,\om_{n-1})|\leq C\del^n
\]
where $C>0$ is some constant, $\om=\{\om_s:\,s\in\bbZ\}$ and $q_{n,z}:\cY^n\to\bbC$ is a family of measurable functions which is uniformly bounded  in $\om_0,\om_1,...,\om_{n-1}$,  $z$ and $n$.
Since $\la_\om(0)=1$ and $\sup_{z\in B(0,\rho)}|\la_{\om}(z)|$ is a bounded random variable, there exist constants $a,b,\rho>0$ so that $P$-a.s. for any $z\in B(0,\rho)$ we have
$|\la_\om(z)|\in(a,b)$. Therefore, also $\Pi_{\om}(z)$ can be approximated exponentially fast in the above sense 
when $z$ belongs to a deterministic neighborhood of $0$.
Thus, using the identifications 
\[
\te^j\om=\{\xi_{s+j}:\,s\in\bbZ\},\,j\in\bbZ
\]
we derive that
for any $k$ there exists a constant $A_k$ and  a uniformly bounded family of functions $g_{k,n}:\cY^n\to\bbC$ so that for any $j\geq0$ and $n\geq1$,
\begin{equation}\label{ExpAppPress}
\left\|\Pi_{\te^j\om}^{(k)}(0)-g_{k,n}(\xi_j,\xi_{j+1},...,\xi_{j+n-1})\right\|_{L^\infty(\Om,\cF,P)}\leq A_k\del^n.
\end{equation}

Next, recall that the $\phi$-mixing (dependence) coefficients associated with the sequence $\{\xi_n: n\in\bbZ\}$ are given by 
\[
\phi(n)=\sup\big\{|P(B|A)-P(B)|: A\in\cF_{-\infty,k},\, B\in{\cF_{k+n,\infty}},\, k\in\bbZ,\, P(A)>0\big\}
\]
where $\cF_{-\infty,s}$ is the $\sig$-algebra generated by $\{\xi_l,\,l\leq s\}$ and $\cF_{s,\infty}$
is the $\sig$-algebra generated by $\{\xi_l: l\geq s\}$ (for any integer $s$).
Recall  (see \cite{Br}, Ch. 4) also that $\phi(n)$ can be written as
\begin{equation}\label{Phi-Rel-StPaper}
\phi(n)=\frac 12\sup\big\{\|\bbE[g|\cF_{-\infty,s}]-\bbE g\|_{\infty}\,:  g\in L^\infty(\Om,\cF_{s+n,\infty},P),\,
\|g\|_{\infty}\leq1,\,s\in\bbZ\big\}.
\end{equation}
We will obtain convergence rates towards the asymptotic moments under the assumption that
\[
S_\phi:=\sum_{n=1}^\infty\phi(n)<\infty.
\]
Under this assumption, the next step of the proof of (\ref{MomRateTemp}) is to approximate $\sum_{j=0}^{n-1}\Pi_{\te^j\om}^{(k)}(0)$ by martingales in the $L^\infty$ norm. The following arguments are classical and are given here for readers' convenience.
For each $m\geq0$, let $\cF_{0,m}$ be the $\sig$-algebra generated by $\xi_0,\xi_1,...,\xi_m$. 
For any $k\geq 2$ and $r\geq 1$ set 
\[
Y_j^{(k,r)}=g_{j,k}(\xi_j,...,\xi_{j+r-1})-\bbE g_{j,k}(\xi_j,...,\xi_{j+r-1}),\,j\geq0
\] 
and let  $M^{(r,k)}_n=\sum_{j=0}^{n-1}X_j^{(k,r)}$ be the martingale 
(w.r.t the filtration $\{\cF_{0,n-1+r}: n\geq 0\}$) whose differences are given by
\[
X_j^{(r,k)}=Y_j^{(k,r)}+\sum_{s\geq j+1}\bbE[Y_s^{(k,r)}|\xi_0,\xi_1,...,\xi_{j+r}]-
\sum_{s\geq j}\bbE[Y_s^{(k,r)}|\xi_0,\xi_1,...,\xi_{j-1+r}].
\]
Then by (\ref{Phi-Rel-StPaper}), similarly to the proof of Theorem 2.8 in \cite{Haf-MD},
we have $\|X_j^{(r,k)}\|_{L^\infty}\leq C(1+r+S_\phi)$ for any $j$, where $C$ is some constant, and, using also (\ref{ExpAppPress}) we derive that
\begin{eqnarray*}
\left\|\sum_{j=0}^{n-1}\Pi_{\te^j\om}^{(k)}(0)-n\int\Pi_{\om'}^{(k)}(0)dP(\om')-M_n^{(k,r)}\right\|_{L^\infty}\\
\leq\left\|\sum_{j=0}^{n-1}Y_j^{(k,r)}-M_n^{(k,r)}\right\|_{L^\infty}+nA_k\del^r\leq 
Q_k(1+r+S_\phi+n\del^r):=R
\end{eqnarray*}
where $Q_k$ is some constant. 
Applying the Azuma-Hoeffding inequality (see, for instance, page 33 in \cite{Mil})  with the martingales $M_n^{(k,r)}$ and $-M_n^{(k,r)}$, we obtain that
that for any $t\geq 0$,
\begin{eqnarray*}
P\big(\big|\sum_{j=0}^{n-1}\Pi_{\te^j\om}^{(k)}(0)-n\int\Pi_{\om'}^{(k)}(0)dP(\om')\big|\geq t+R\big)
\leq P(|M_n^{(k,r)}|\geq t)\leq P(M_n^{(k,r)}\geq t)\\+P(-M_n^{(k,r)}\geq t)
=P(e^{\la_t M_n^{(k,r)}}\geq e^{t\la_t})+(e^{\la_t\cdot(-M_n^{(k,r)})}\geq e^{t\la_t})\\
\leq e^{-t\la_t}\big(\bbE e^{\la_t M_n^{(k,r)}}+\bbE  e^{-\la_t M_n^{(k,r)}}\big)
\leq 2 e^{-t\la_t}e^{\la_t^2\sum_{j=0}^{m-1}\|X_j^{(k,r)}\|_{L^\infty}}
\leq 2e^{-\frac{t^2}{4nC(1+r+S_\phi)}}
\end{eqnarray*}
where  $\la_t=\frac{t}{2nC(1+r+S_\phi)}$, and in the third inequality we used the Markov inequality.
Taking $r=r_n$ of logarithmic order in $n$, we derive  that there exist constants $a_k,C_k,c_k,d_k>0$ so that for any $n\geq2$ and $\ve>0$,
\[
P\big\{\om: |\Pi_{\om,n}^{(k)}(0)-n\int \Pi_{\om'}^{(k)}(0)dP(\om')|\geq n\ve+a_k\ln n \big\}\leq C_ke^{-c_k\frac{\ve^2 n}{\ln n}}.
\]
By taking $\ve=u n^{-\frac12}\ln n$ for a sufficiently large $u$, we derive from the Borel Cantelli Lemma that for any $k$, $P$-a.s for any sufficiently large $n$ we have
\begin{equation}\label{P rate}
\left|\frac{1}{n}\Pi_{\om,n}^{(k)}(0)-\int \Pi_{\om'}^{(k)}(0)dP(\om')\right|\leq b_k n^{-\frac12}\ln n
\end{equation}
where $b_k>0$ is some constant which depends only on $k$.
It follows from (\ref{P rate}), (\ref{Rate even}) and (\ref{Rate odd}) that the convergence rate towards the 
asymptotic moments is at most of order $n^{-\frac12}\ln n$ (i.e. that (\ref{MomRateTemp}) holds true).

Next, the arguments in the proof of Theorem 7.1.1 in \cite{book}  show  that when $\sig^2>0$ then there exist random variables $c_\om$ and $d_\om$ so that $P$-a.s. for any $n\geq1$,
\begin{eqnarray*}
\sup_{s\in\bbR}\big|\mu_\om\{x\in\cE_\om: S_n^\om(x)\leq s\sqrt n\}-\frac1{\sqrt{2\pi\sig^2}}\int_{-\infty}^s e^{-\frac{t^2}{2\sig^2}}dt\big|\\\leq 
c_\om n^{-\frac12}+d_\om\big|\frac 1n\mu_\om(S_n^\om)^2-\sig^2\big|.
\end{eqnarray*}
Therefore, under the above mixing conditions we derive almost optimal convergence rate in the central limit theorem of the form
\[
\sup_{s\in\bbR}\left|\mu_\om\{x\in\cE_\om: S_n^\om(x)\leq s\sqrt n\}-\frac1{\sqrt{2\pi\sig^2}}\int_{-\infty}^s e^{-\frac{t^2}{2\sig^2}}dt\right|\leq 
q_\om n^{-\frac12}\ln n
\]
where $q_\om$ is some random variable (which in general may not even be integrable).




\begin{remark}
When $\{\xi_n:\,n\geq1\}$ is a stationary geometrically ergodic Markov chain then $\phi(n)$ converges exponentially fast to $0$ (see \cite{Br}, Theorem 21.1), and in particular $S_\phi<\infty$. Such exponential convergence occurs also when $\xi_n$ has form $\xi_n(x)=\xi_0\circ S^n(x)$, where $S$ is a two sided topologically mixing subshift of finite type, $x$ is distributed according to some  Gibbs measure (see \cite{Bow}) and $\xi_0$ is a function of a finite number of coordinates.
If $\{\zeta_n: n\geq0\}$ is a one sided stationary process then we can define $\xi_n=(\xi_{n,j})_{j=-\infty}^{\infty}=(...,\zeta_{n-1},\zeta_n,\zeta_{n+1},...)$, where $\zeta_n$ appears in the $0$-the coordinate. In these circumstances, we can write $\om=(\om_j)_{j\in\bbZ}$ and $\om_j=(\om_{j,k})_{k=-\infty}^\infty$.
When $\cA_z^\om$ depends only on $\om_{0,0}$ then all the arguments from this section hold true if we assume that the sequence $\phi_1(n),\,n\geq1$ given by 
\[
\phi_1(n)=\sup\big\{|P(B|A)-P(B)|: A\in\cG_{-\infty,k},\, B\in{\cG_{k+n,\infty}},\, k\in\bbN,\, P(A)>0\big\}
\]
is summable, where $\cG_{-\infty,s}$ is the $\sig$-algebra generated by $\{\zeta_0,\zeta_1,...,\zeta_s\}$ and $\cG_{s,\infty}$ is the $\sig$-algebra generated by $\{\zeta_s,\zeta_{s+1},...\}$. Note also that the above results hold true also when
 $\xi_n$ has the form $\xi_n(x)=\xi_0\circ T^n(x)$, where $T$ is a Young tower (see \cite{Young1} and \cite{Young2}) whose tails decay to $0$ sufficiently fast,  $x$ is distributed according to an appropriate Gibbs measure and $\xi_0$ is a bounded random variables which is measurable with respect to the partition which defines the tower. Young towers do not seem to be $\phi$-mixing with respect to the natural filtrations generated by cylinder sets, but they  satisfy certain type of ``graded $\phi$-mixing" conditions which allows use to obtain the results above (see Section 7 in \cite{Haf-Arrays} for several mixing properties of Young towers which involve certain $\phi$-mixing coefficients).
\end{remark}

\begin{remark}\label{RPF trip LimThms}
The argument above show that the RPF triplets $\la_\om(z), h_\om(z)$ and $\nu_\om(z)$ and all their derivatives with respect to $z$ can be approximated by functions of the random variables $\xi_0,\xi_1,...,\xi_n$ (or $\xi_{-n},...,\xi_0$ in the case of $ h_\om(z)$) with an error term which converges to $0$ exponentially fast when $n\to\infty$. Therefore, applying classical results from probability theory (for stationary and mixing processes) we obtain that for any real $t$ with a sufficiently small absolute value, the partial sums $\sum_{j=0}^{n-1}X_{i,k}(\te^j\om)$ generated by the random variables $X_{1,k}(\om)=\la_\om^{(k)}(t)$,  $X_{2,k}(\om)=\big(h_\om^{(k)}(t)\big)(x)$ or $X_{3,k}(\om)=\big(\nu^{(k)}_\om(t)\big)(g)$ satisfy exponential concentration inequalities, as well as various limit theorems such as the CLT. Here we assume that $\cE_\om=\cX$ does not depend on $\om$, $x\in\cX$ and $g$ is a H\"older continuous function on $\cX$. 
\end{remark}


\section{Edgeworth type expansions: Proof of Theorem \ref{EdgeThm}}\label{EdgeSec}\setcounter{equation}{0}
The proof of Theorem \ref{EdgeThm} relies on classical arguments involving Fourier transforms which were used
 successfully in the i.i.d. case and in the deterministic case (i.e. when $|\Om|=1$), 
see, for instance, \cite{Feller}, \cite{Liverani} and references therein. Most of the arguments in the proof of Theorem \ref{EdgeThm} are modifications of the arguments in \cite{Liverani}. When the arguments are exactly as in \cite{Liverani}, we will just refer the reader's  there. The main difficulty is to verify Assumptions \ref{BaseAss}, \ref{A3} and \ref{A4}, and this is postponed to Section \ref{secEXM}.

For reader's convenience, we first will describe the main idea behind using Fourier transforms in order to derive Edgeworth expansions. 
Recall first that by the
Esseen inequality (see \cite{Lin}, Ch. 4.1), there exists an absolute constant $A$ so that for any  $T>0$, a distribution function $F$ and an integrable function $G:\bbR\to\bbR$ with bounded first derivative such that $\lim_{x\to\infty} G(x)=0$,
\begin{equation}\label{Esseen}
\sup_{x\in\bbR}|F(x)-G(x)|\leq \int_{-T}^T\Big|\frac{f(t)-g(t)}{t}\Big|dt+\frac{A}{T}\sup_{x\in\bbR}|G'(x)|
\end{equation}
where $f(t)=\int e^{itx}dF(x)$ and $g(t)=\int e^{itx}dG(x)=\hat{G'}(-t)$ (here $\hat g$ is the Fourier transform of  a function $g$).  Let $F(s)=\mu_\om\{x:\in\cE_\om: S_n^\om\leq\sqrt n s\}$. The idea behind the proof is to find a
function $H$ so that $\frac{|\hat F(t)-H(t)|}{|t|}$
will be sufficiently small on some interval $[-T,T]$, where 
$T$ is of order $n^{-\frac{d}{2}}$. In this case we could take $G(s)=\int_{-\infty}^s\hat H(x)dx$ (so $\hat G'(-t)=H(t)$) and it will remain to verify that $G$ has the desired form.
On intervals of the form $I_n(\del)=[-\del\sqrt n,\del\sqrt n]$ for small $\del$'s this approximation is done using Assumption \ref{BaseAss}, while on 
$J_n(\del,a)=[-a\sqrt n, a\sqrt n]\setminus I_n(\del)$ (for large $a$'s) we will use Assumption \ref{A3}. When $d>1$ we will use Assumption \ref{A4} in order to obtain an appropriate estimate on $[-T,T]\setminus J_n(a,\del)$.



Let $d\geq1$. The first step of the proof of Theorem \ref{EdgeThm} is as follows. The arguments preceding (\ref{II}) show that for any $z\in B(0,\frac \rho 2)$ we can write
\begin{equation}\label{II.1}
\mu_\om(e^{zS_n^\om})=e^{\sum_{j=0}^{n-1}\Pi_{\te^j\om}(z)}W_{\te^n\om}(z)+\la_{\om,n}(z)\del_{\om,n}(z)
\end{equation}
where $W_{\om}(z)=\nu_{\om}(z)(h_{\om}(z))$, $d_{\om,n}(0)=0$
and $|\del_{\om,n}^{(k)}(z)|\leq C_k\del^n$. Since $\del_{\om,n}$ is analytic and $\del_{\om,n}(0)=0$, it follows that $|\del_{\om,n}(z)|\leq C'_0|z|\del^n$ for $z$'s in the above domain, where $C_0'$ is some constant. Moreover, it follows from Assumption \ref{A3} (i) together with (\ref{ExpConvAss}) that $P$-a.s. we have 
$\sup\{|\la_{\om,n}(it)|:\,t\in(-r,r),n\geq1\}\leq C_0''c_\om$, where $c_\om$ is some random variable. Therefore,
$P$-a.s. for any $|t|<\frac \rho 2$ and $n\geq1$,
\begin{equation}\label{First}
|\la_{\om,n}(it)\del_{\om,n}(it)|\leq C_\om|t|\del^n
\end{equation}
for some random variable $C_\om$.
Next, since $\Pi_{\om}(0)=\Pi_\om'(0)=0$ and all the derivatives of $\Pi_\om(z)$ are uniformly bounded in $\om$, where $\om$ ranges over a set of probability $1$,
we can write
\[
\sum_{j=0}^{n-1}\Pi_{\te^j\om}(\frac {it}{\sqrt n})=-\frac{t^2\Pi_{\om,n,2}}{2}+n\psi_{\om,n}(\frac {t}{\sqrt n})
\] 
where 
\[
\Pi_{\om,n,2}=n^{-1}\sum_{j=0}^{n-1}\Pi_{\te^j\om}^{(2)}(0)\,\,\text{ and }\,\,\psi_{\om,n}(z)=\frac 1n\sum_{j=0}^{n-1}(\Pi_{\te^j\om}(iz)-\frac{(iz)^2}2\Pi_{\te^j\om}''(0)).
\] 
By (\ref{II.1}) and (\ref{First}) we have for $P$-a.a. $\om$, 
\begin{equation}\label{II.2}
\mu_\om(e^{\frac{itS_n^\om}{\sqrt n}})=\exp\big(-\frac{t^2\Pi_{\om,n,2}}{2}+n\psi_{\om,n}(\frac{t}{\sqrt n})\big)W_{\te^n\om}(\frac{it}{\sqrt n})
+|t|o(n^{-\frac{d}{2}}).
\end{equation} 
Let $k\geq0$. Then, by Theorem \ref{EdgeThm} (i),  there exists a constant $B_k>0$ so that $P$-a.s. we have $\sup_{z\in B(0,\rho)}|\psi_{\om,n}(z)|\leq B_k$ for any $n\geq1$. Observe also that $\psi_{\om,n}(0)=\psi_{\om,n}'(0)=\psi_{\om,n}''(0)=0$. Therefore we can write $\psi_{\om,n}(z)=z^2\psi_{\om,n,d}(z)+|z|^{d+2}\tilde\psi_{\om,n,d}(z)$, where $\psi_{\om,n,d}$ is a polynomial of degree $d$ and $\tilde\psi_{\om,n,d}$ is an analytic function which vanishes at $z=0$ and is bounded in some deterministic neighborhood of the origin by some constant which does not depend on $\om$ and $n$. Using the latter notations, we can write 
\[
\exp\big(n\psi_{\om,n}(\frac{it}{\sqrt n})\big)=\exp\left(t^2\psi_{\om,n,r}(\frac{t}{\sqrt n})+
n^{-\frac d2}t^{d+2}\tilde\psi_{\om,n,d}(\frac{t}{\sqrt n})\right).
\]
Since $\mu_\om(h_\om(0))$, $\mu_\om$ is a probability measure, $h_\om(z)$ is analytic in $z$ and $\sup_{z\in B(0,\rho)}\|h_\om(z)\|_1$ is a bounded random variable, there exists a constant $0<r_0<\frac \rho 2$ so that $P$-a.s. we can develop a branch of $\ln W_{\om}(z)$ in $B(0,r_0)$ which is analytic, uniformly bounded in $\om$ (when $\om$ ranges over a set of probability $1$) and takes the value $\ln W_\om(0)=0$ at $z=0$. 
Next, we can write $W_{\om}(it)=1+W_{\om,d}(t)+t^d\tilde W_{\om,d}(t)$, where $W_{\om,d}(t)$ is a polynomial of degree $d$ which vanishes at $t=0$, whose coefficients are bounded random variables,
and $\tilde W_{\om,d}(t)$ is a $C^\infty$ function which vanishes at $t=0$, 
whose derivatives are uniformly bounded in $\om$ (around $0$). Similarly to the derivation of equality  (3.7) in \cite{Liverani}, 
and using the Taylor expansions of the functions $\exp(\cdot)$ and $\ln(\cdot)$ we derive that $P$-a.s. for any 
$t\in(- n^{\frac 12}r_0,n^{\frac 12}r_0)$,

\begin{eqnarray}\label{(3.7)}
\exp\big(\frac{t^2\Pi_{\om,n,2}}{2}+n\psi_{\om,n}(\frac{it}{\sqrt n})\big)W_{\te^n\om}(\frac{it}{\sqrt n})\\=
\exp\big(\frac{t^2\Pi_{\om,n,2}}{2}+n\psi_{\om,n}(\frac{it}{\sqrt n})+\ln e^{W_{\te^n\om}(\frac{it}{\sqrt n})}\big)\nonumber\\=
\exp\Big(t^2\psi_{\om,n,d}(\frac{it}{\sqrt n})+\frac1{n^{d/2}}t^{d+2}\tilde\psi_{\om,n,d}(\frac{it}{\sqrt n})-\nonumber\\\sum_{k=1}^{d}\frac{(-1)^{k+1}}{k}\big(W_{\om,d}(\frac{it}{\sqrt n})\big)^k-\frac1{n^{d/2}}t^{d}\overline {W}_{\om,d}(\frac{it}{\sqrt n})\Big)\nonumber\\=
1+\sum_{m=1}^d\frac1{m!}\Big(t^2\psi_{\om,n,d}(\frac{it}{\sqrt n})-\sum_{k=1}^d \frac{(-1)^{k+1}}{k}\big(W_{\om,d}(\frac{it}{\sqrt n})\big)^k\Big)^m \nonumber\\+\frac1{n^{d/2}}t^{d+2}\tilde\psi_{\om,n,d}(\frac{it}{\sqrt n})-\frac1{n^{d/2}}t^{d}\overline{W}_{\om,d}(\frac{it}{\sqrt n})+t^{d+1}\cO(n^{-\frac{d+1}{2}})\nonumber\\=
\sum_{k=0}^d n^{-\frac k2}A_{\om,n,k}(t)+
t^d n^{-\frac d2}\varphi_{\om,n}(\frac{t}{\sqrt n})+t^{d+1}\cO(n^{-\frac{d+1}{2}}).\nonumber
\end{eqnarray}
Here $A_0\equiv1$ and all the other $A_{\om,n,k}$'s are polynomials of degree $s_k$ which does not depend on $\om$ and $n$, whose coefficients are algebraic combinations of the  derivatives of $W_{\om,n}$ at $0$ and the derivatives of $\psi_{\om,n}$ at $0$, and $\varphi_{\om,n}(z)=z^2\tilde\psi_{\om,n,d}(z)-\overline{W}_{\om,n,d}(z)$, where $\overline{W}_{\om,n,d}(z)$ is the reminder of $\ln W_{\om,n}(iz)$ when approximated by powers of $W_{\om,n,d}(iz)$. 
Note that $\varphi_{\om,n}(z)$ is analytic, that it vanishes at $z=0$ and that $\sup_{a\in B(0,r_1)}|\varphi_{\om,n}(z)|\leq C_1$ ($P$-a.s.) for some positive constants $r_1$ and $C_1$ which do not depend on $\om$.
Set
\[
Q_{\om,n}(t)=\sum_{k=1}^d n^{-\frac k2}A_{\om,n,k}(t).
\]
Since $\Pi_{\om,n,2}$ converges to $\sig^2>0$ as $n\to\infty$,
it follows from (\ref{(3.7)}) that for any sufficiently small $\del_0$, $P$-a.s. we have
\begin{eqnarray}\label{(3.10)}
\int_{-\del_0\sqrt n}^{\del_0\sqrt n}\Big|\frac{\la_{\om,n}(\frac{it}{\sqrt n})W_{\te^n\om}(\frac{it}{\sqrt n})-e^{-\frac{t^2\Pi_{\om,n,2}}{2}}(1+Q_{\om,n})(t)}t\Big|dt\\=
\int_{-\del_0\sqrt n}^{\del_0\sqrt n}e^{-\frac{t^2\Pi_{\om,n,2}}{2}}\Big|\frac{\exp\big(n\psi_{\om,n}(\frac{it}{\sqrt n})+\ln W_{\te^n\om}(\frac{it}{\sqrt n})\big)-1-Q_{\om,n}(t)}{t}\Big|dt=o(n^{-\frac d2}).\nonumber
\end{eqnarray}
Combining this with (\ref{II.2}), we derive that for any sufficiently small $\del_0>0$ for $P$-a.a. $\om$,
\begin{equation}\label{3.12}
\int_{-\del_0\sqrt n}^{\del_0\sqrt n}\Big|\frac{\mu_\om(e^{\frac{it S_n^\om}{\sqrt n}})-e^{-\frac{t^2\Pi_{\om,n,2}}{2}}(1+Q_{\om,n}(t))}{t}\Big| dt=o(n^{-\frac d2}).
\end{equation}
Consider the functions $g_{\om,n}(t)=e^{-\frac{t^2\Pi_{\om,n,2}}{2}}(1+Q_{\om,n}(t))$. 
Using the fact that for any $c>0$, the function $\sqrt{2\pi c}(it)^ke^{-\frac{t^2c}{2}}$ is the  Fourier transform of the function $e^{-\frac{t^2}{2c}}$, and using the integration by parts formula with integrals of the form 
$\int x^je^{-cx^2}dx$,
we derive that 
\begin{equation}\label{Finding Poly1}
\sqrt{2\pi \Pi_{\om,n,2}}g_{\om,n}=\textbf{F}\Big(\int_{-\infty}^te^{-\frac{x^2}{2\Pi_{\om,n,2}}}dx+\sum_{k=1}^d n^{-\frac{k}{2}}R_{\om,n}(t)e^{-\frac{t^2}{2\Pi_{\om,n,2}}}\Big)
\end{equation}
where $\textbf{F}$ is the Fourier transform operator and
 all the $R_{\om,n,k}$'s are polynomials whose coefficients
are linear combinations of the coefficients of the $A_{\om,n,k}$'s, and the degree of $R_{\om,n,k}$ depends only on $k$. Define $G_{\om,n}$ by $G_{\om,n}(t)=\int_{-\infty}^t\textbf F\big(g_{\om,n}\big)(-\xi)d\xi$. Then by the Fourier inversion theorem, $\textbf{F}(G'_{\om,n})=g_{\om,n}$ and $\xi\to\sqrt{2\pi \Pi_{\om,n,2}}\textbf{F}(g_{\om,n})(-\xi)$ is the function inside the Fourier transform on the right hand side of (\ref{Finding Poly1}). 
Using also the integration by parts formula with integrals of the form $\int t^je^{-ct^2}dt=\int t^{j-1}te^{-ct^2}dt$, we derive that
\begin{equation}\label{Finding Poly1.1}
\sqrt{2\pi \Pi_{\om,n,2}}G_{\om,n}(t)=\int_{-\infty}^te^{-\frac{x^2}{2\Pi_{\om,n,2}}}dx+\sum_{k=1}^d n^{-\frac{k}{2}}P_{\om,k,n}(t)e^{-\frac{t^2}{2\Pi_{\om,n,2}}}
\end{equation}
where all the $P_{\om,n,k}$'s are polynomials whose coefficients
are linear combinations of the coefficients of the $A_{\om,n,k}$'s, and the degree of $P_{\om,n,k}$ depends only on $k$.

Next, applying (\ref{Esseen}) with  $F=F_{\om,n}(s)=\mu_\om\{x\in\cE_\om: S_n^\om\leq\sqrt n s\}$ and $G=G_{\om,n}$,  taking into account (\ref{3.12}), we obtain that for any $\ve>0$ and $B>\frac{A}\ve$, where $A$ comes from (\ref{Esseen}),
\begin{eqnarray*}
\sup_{s\in\bbR}|F_{\om,n}(s)-G_{\om,s}(s)|\leq \int_{-Bn^{\frac d2}}^{Bn^{\frac d2}}\Big|\frac{\mu_\om(e^{\frac{it S_n^\om}{\sqrt n}})-e^{-\frac{t^2\Pi_{\om,n,2}}{2}}(1+Q_{\om,n}(t))}t\Big|+\frac{A}{Bn^{\frac d2}}\\\leq I_1+I_2+I_3+\ve n^{-\frac d2}
\end{eqnarray*}
where 
\begin{eqnarray*}
I_1=\int_{-\del_0\sqrt n}^{\del_0\sqrt n}\Big|\frac{\mu_\om(e^{\frac{it S_n^\om}{\sqrt n}})-e^{-\frac{t^2\Pi_{\om,n,2}}{2}}(1+Q_{\om,n}(t))}t\Big|dt=o(n^{-\frac d2}),\\
I_2=\int_{\del_0\sqrt n<|t|<Bn^{\frac d2}}\Big|\frac{\mu_\om(e^{\frac{itS_n^\om}{\sqrt n}})}t\Big|dt\,\,\text{ and }\,\,\\
I_3=\int_{|t|>\sqrt n\del_0}e^{-\frac{t^2\Pi_{\om,n,2}}2}\big|t^{-1}(1+Q_{\om,n}(t))\big|dt.
\end{eqnarray*}
We note that the polynomials $P_{\om,k,n}$ appearing in the definition of $G$ depend on $n$.
Therefore, in principle, one cannot apply (\ref{Esseen}) directly. One can do it when $G'_{\om,n}$ is uniformly bounded in $n$ and $\lim_{x\to\infty}\sup_{n\geq1}|G_{\om,n}(x)|=0$ (these conditions are clearly satisfied in our circumstances).
In (\ref{3.12}) we have shown that $I_1=o(n^{-\frac d2})$. Since $\Pi_{\om,n,2}$ converges to $\sig^2>0$ and the coefficients of all $A_{\om,k}$'s in the definition of $Q_{\om,n}$ are bounded in $n$, it follows that $I_3=\cO(e^{-cn})$ for some $c>0$. In order to complete the proof of Theorem \ref{EdgeThm}, it is sufficient to show that $I_2=o(n^{-\frac d2})$. Recall that 
\[
\mu_\om(e^{it S_n^\om})=\mu_{\te^n\om}(\cA_{it}^{\om,n}\textbf{1})
\]
for any $t\in\bbR$.
In the case when $d=1$, using Assumption \ref{A3} with $r=\frac d2$ and the compact set 
$J=\{t\in\bbR: \del_0\leq |t|\leq B\}$ we obtain that 
\[
I_2\leq \frac{1}{\sqrt n\del_0}\sup_{t\in J}\|\cA_{it}^{\om,n}\|_1=\cO(n^{-\frac d2})
\]
for some $c>0$, which completes the proof of Theorem \ref{EdgeThm} (i). When $d>1$, in order to complete the proof of Theorem \ref{EdgeThm} (ii), it remains to 
show that, $P$-a.s. for any sufficiently large $C>0$ 
\[
\int_{C\sqrt n<|t|<Bn^{\frac d2}}\Big|\frac{\mu_\om(e^{\frac{itS_n^\om}{\sqrt n}})}t\Big|dt=o(n^{-\frac d2}).
\]
This follows from the following lemma, together with the estimate 
$|\mu_\om(e^{itS_n^\om})|\leq\|\cA_{it}^{\om,n}\|_2$, where $\|\cdot\|_2$ is the norm specified in Assumption \ref{A4}.
\begin{lemma}
Suppose that Assumption \ref{A4} holds true with $r=\frac{d}2$ and that $(\Om,\cF,P,\te)$ is ergodic. Then for any $q>0$ there exist positive constants $K,D$ and $L$ so that $P$-a.s. for any sufficiently large $n$, and $t$ such that $K\leq |t|\leq Dn^{r}$,
\[
\|\cA_{it}^{\om,n}\|_2\leq LR_\om n^{-q}
\]
where $R_\om$ is the random variable specified in Assumption \ref{A4}.
\end{lemma}

\begin{proof}
Set $r=\frac d2>0$ let $K_\om,D_\om,C_\om$, $r_2(\om)$, $R_\om$  and $N(\om)$ be the random variables specified in Assumption \ref{A4} with this specific $r$. Let $K,D,r_2,C$ and $N$ be positive numbers so that the set 
\[
\Gam=\{\om\in\Om:\,K_\om\leq K,\,D_\om\geq D,\,r_2(\om)\geq r_2,\,C_\om\leq C\,\text{ and }\,N(\om)\leq N\}
\]
has positive probability. Set $p_0=P(\Gam)$. By Birkhoff's ergodic theorem and Kac's lemma applied with the transformation induced on $\Gamma$ by $\te$, we obtain that  $P$-a.s. there exists a sequence $m_1<m_2<m_3<...$ of positive integers so that 
\[
\te^{m_i}\om\in\Gam\,\text{ for all }\,i's\,\text{ and }\,\lim_{k\to\infty}\frac{m_k}{k}=\frac1{p_0}.
\]
It follows that for any $0<a<b<1$ and a sufficiently large $n$ there exists an index $k=k_{\om,n,a,b}$ so that 
$an<m_k<bn$. Indeed, let $1>\ve>0$. Then for any sufficiently large $k$ we have 
$\frac{1-\ve}{p_0}k<m_k<\frac{1+\ve}{p_0}k$. Suppose that $\ve$ is sufficiently small and let
 $\del>0$ be so that $ap_0<\del(1-\ve)<\del(1+\ve)<bp_0$. Let 
$(k_n)_{n=1}^\infty$ be a strictly increasing sequence of positive integers so that $k_n/n$ converges to $\del$ as $n\to\infty$. Then for any sufficiently large $n$ we have 
\[
an<\frac{1-\ve}{p_0}k_n<m_{k_n}<\frac{1+\ve}{p_0}k_n<bn.
\]

Next, let $q>0$ and let $m$ be a positive integer such that $mr_2>r+q$. For any sufficiently large $n$ we can find $k_1(n),k_2(n),...,k_m(n)$ so that for each $i$,
\[
\big(\sum_{j=1}^{i}2^{-j}+2^{-i-2}\big)n>m_{k_i(n)}>\big(\sum_{j=1}^{i}2^{-j}\big)n.
\]
Set $\om_i=\te^{m_i}\om$, $\Del_i(n)=k_{i+1}(n)-k_{i}(n)$, $i=1,2,...,m-2$, $\Del_0(n)=k_1(n)-m_1$, 
$\Del_{m-1}(n)=n-k_{m-1}(n)$ and $k_0(n)=1$. Then there exists a constant $c>0$ which depends only 
on $m$ so that 
\[
\Del_i(n)\geq cn\,\,\text{ for any }\,i=0,1,...,m-1.
\]
Now, $P$-a.s for 
any large enough $n$ we can write for all $t$'s,
\[
 \cA_{it}^{\om,n}=\cA_{it}^{\om_{k_{m-1}(n)},\Del_{m-1}(n)}\circ\cdots\circ\cA_{it}^{\om_{k_1(n)},\Del_1(n)}\circ\cA_{it}^{\om_{k_n(0)},\Del_0(n)}\circ\cA_{it}^{\om,m_1}.
\]
If $K\leq |t|\leq Dn^d$ then, since $\om_i\in\Gam$ for all $i$'s, for any sufficiently large $n$ and $0\leq i\leq m-1$ we have 
\[
\|\cA_{it}^{\om_{k_i(n)},\Del_i(n)}\|_2\leq C\big(\Del_i(n)\big)^{-r_2}\leq
Cc^{-r_2}n^{-r_2}.
\]
By Assumption \ref{A4} we have $\|\cA_{it}^{\om,m_1}\|_2\leq R_\om(1+|t|)$ which is less than 
$R_\om(1+ Dn^r)$ for $t$'s in the above range. We conclude from submultiplicity of norms of operators 
that $P$-a.s. for any sufficiently large $n$ 
and $t$ so that $K\leq |t|\leq Dn^r$ we have 
\[
\|\cA_{it}^{\om,n}\|_2\leq C^mc^{-mr_2}(1+D)R_\om n^{-mr_2+r}\leq 
 C^mc^{-mr_2}(1+D)R_\om n^{-q}:=LR_\om n^{-q}
\]
and the proof of the lemma is complete.
\end{proof}

Now we will explain how to derive Theorem \ref{EdgeThm} (iii). 
It follows from the definition of the $P_{\om,n,k}$'s that their coefficients are algebraic combinations of the derivatives at $z=0$
of the functions $\mu_\om(h_\om(z))$ and $n^{-1}\sum_{j=0}^{n-1}\Pi_{\te^j\om}(z)$.  
Since the derivatives of these functions $0$ are uniformly bounded in $\om$ and $n$ (when $\om$ ranges over a set of full probability), these coefficient are uniformly bounded in $\om$ and $n$, and the coefficients of $P_{\te^{-n}\om,n,k}$ converge $P$-a.s. and in $L^p$ for any $p\in[1,\infty)$. When 
 (\ref{P rate}) holds true then all the coefficients of the polynomials $P_{\te^{-n}\om,n,k}, k=1,2,...$ converge towards their limits with rate of order $n^{-\frac12}\ln n$.





\section{Examples}\setcounter{equation}{0}\label{secEXM}

\subsection{Random distance expanding maps}\label{SkewProdSec}
Let $(\Om,\cF,P,\te)$, $(\cX,\rho)$ and $\cE_\om$ be as described at the beginning of Section \ref{sec2}, 
and let
$
\{T_\om: \cE_\om\to \cE_{\te\om},\, \om\in\Om\}
$
be a collection of  maps between the metric spaces 
$\cE_\om$ and $\cE_{\te\om}$, so that
the map $(\om,x)\to T_\om x$ is measurable with respect to the $\sigma$-algebra $\cP$
which is the restriction of $\cF\times\cB$ on $\cE$.
Consider the 
skew product transformation $T:\cE\to\cE$ given by 
\begin{equation}\label{Skew product}
T(\om,x)=(\te\om,T_\om x).
\end{equation}
For any $\om\in\Om$ and $n\in\bbN$ consider the $n$-th step iterates $T_\om^n$ given by
\begin{equation}\label{T om n}
T_\om^n=T_{\te^{n-1}\om}\circ\cdots\circ T_{\te\om}\circ T_\om: \cE_\om\to\cE_{\te^n\om}.
\end{equation}
Our additional requirements concerning the family of maps $\{T_\om:\om\in\Om\}$
are collected in the following assumptions.
\begin{assumption}[Topological exactness]\label{TopExAssRand}
There exist a constant $\xi>0$ and a random variable $n_\om\in\bbN$ such that 
$P$-a.s.,
\begin{equation}\label{TopExRand}
T_\om^{n_\om}(B_\om(x,\xi))=
\cE_{\te^{n_\om}\om}\,\text{ for any } x\in\cE_\om
\end{equation}
where for any $\om\in\Om$, $x\in\cE_\om$ and $r>0$, 
$\,B_\om(x,r)$ denotes a ball in $\cE_\om$ around $x$ with radius $r$.
\end{assumption}

\begin{assumption}[The pairing property]\label{Ass pair prop}
There exist random variables $\gam_\om>1$ and $D_\om\in\bbN$ such that $P$-a.s.
for any $x,x'\in\cE_{\te\om}$ with $\rho(x,x')<\xi$ 
we can write 
\begin{equation}\label{Pair1.0}
T_\om^{-1}\{x\}=\{y_1,...,y_k\}\,\,\text{ and }\,\,T_\om^{-1}\{x'\}=\{y_1',...,y_k'\}
\end{equation}
where $\xi$ is specified in Assumption \ref{TopExAssRand},
\[
k=k_{\om,x}=|T_\om^{-1}\{x\}|\leq D_\om 
\]
and
\begin{equation}\label{Pair2.0}
\rho(y_i,y_i')\leq (\gam_\om)^{-1}\rho(x,x')
\end{equation}
for any $1\leq i\leq k$.
\end{assumption}
These assumptions hold true for the random distance expanding maps considered in \cite{MSU} since in the setup there the inverse branches are contracting, see Section 2.5 in \cite{MSU}.
We refer the readers to Section 2.1 in \cite{MSU} for a  description of several examples of such maps $T_\om$.

According to Lemma 4.11 in \cite{MSU} (applied with $r=\xi$), 
there exist an integer valued random variable $L_\om\geq1$ and $\cF$-measurable 
functions $\om\to x_{\om,i}\in\cX,\,i=1,2,3,...$ so that 
$x_{\om,i}\in\cE_\om$ for each $i$ and
\begin{eqnarray}\label{cover}
\bigcup_{k=1}^{L_\om} B_\om(x_{\om,k},\xi)=\cE_\om,
\,\,\,P\text{-a.s.}
\end{eqnarray} 
Note that $L_\om$ is constant in $\om$ when
$\cE_\om$ does not depend on $\om$ (i.e. when $\cE=\Om\times\cY$ for an appropriate
$\cY\subset\cX$).

Next, for any $g:\cE\to\bbC$ and $\om\in\Om$ consider 
the function $g_\om:\cE_\om\to\bbC$ given by $g_\om(x)=g(\om,x)$.
Let $\phi, u:\cE\to\bbR$ be measurable functions so that for $P$-a.a. $\om$ the functions 
$\phi_\om$ and $u_\om$ are H\"older continuous 
with exponent $\al\in(0,1]$ which is 
independent of $\om$.
Let $z\in\bbC$ and consider the transfer operators $\cL_z^\om,\,\om\in\Om$ which 
map functions on $\cE_\om$ to functions on $\cE_{\te\om}$ by
the formula\index{transfer operator!random!complex}
\begin{eqnarray}\label{TraOp}
\cL^\om_z g(x)=\sum_{y\in T_\om^{-1}\{x\}}e^{\phi_\om(y)+zu_\om(y)}g(y).
\end{eqnarray}
Note that under Assumption \ref{Ass pair prop} the operators 
$\cL^\om_z,\,z\in\bbC$ are well defined and since $\phi_\om$ and $u_\om$ are H\"older continuous
they map a continuous function on $\cE_\om$
to a continuous function on $\cE_{\te\om}$. 
For any $n\in\bbN$ and $z\in\bbC$ consider the $n$-th 
step iterates $\cL_z^{\om,n}$ of the transfer operator given by
\begin{equation}\label{iter}
\cL_z^{\om,n}=\cL_z^{\te^{n-1}\om}\circ\cdots\circ\cL_z^{\te\om}\circ\cL_z^\om.
\end{equation}
Then 
\begin{equation}\label{iter2}
\cL_z^{\om,n}g(x)=\sum_{y\in (T_\om^n)^{-1}\{x\}}e^{S_n^\om \phi(y)+zS_n^\om u(y)}g(y)
\end{equation} 
where
$
S_n^\om \psi=\sum_{i=0}^{n-1}\psi_{\te^i\om}\circ T_\om^i
$
for any function $\psi:\cE\to\bbC$, $\om\in\Om$ and $n\geq1$.
We also consider the (global) transfer operator $\cL_z$
acting on functions  $g:\cE\to\bbC$ by the formula 
\begin{equation}\label{TraOp1}
\cL_z g(s)=
\sum_{s'\in T^{-1}\{s\}=(\te^{-1}\om,y)}e^{\phi_{\te^{-1}\om}(y)+zu_{\te^{-1}\om}(y)}g(s')=
\cL^{\te^{-1}\om}_z g_{\te^{-1}\om}(x),\,\,\,s=(\om,x),
\end{equation}
namely $\cL_z$  is generated by the skew product map $T$ and the function
$\phi+zu$.
Next, let $g:\cE\to\bbC$ be a measurable function. Let 
$\al$ as described before (\ref{TraOp}) and set
\begin{eqnarray*}
v_{\al,\xi}(g_\om)=\inf\{R: |g_\om(x)-g_\om(x')|\leq R\rho^\al(x,x')\,\text{ if }\,
\rho(x,x')<\xi\}\\
\text{and }\,\,\,\|g_\om\|_1=\|g_\om\|_{\al,\xi}=\|g_\om\|_\infty+v_{\al,\xi}(g_\om)\hskip1cm
\end{eqnarray*}
where $\|\cdot\|_\infty$ is the supremum norm and $\rho^\al(x,x')=\big(\rho(x,x')\big)^\al$. 
These norms are $\cF$-measurable as a consequence
of Lemma 5.1.3 in \cite{book}. We denote by $\cH_{\al,\xi}^\om=H_1^\om$ the space of all functions 
$f:\cE_\om\to\bbC$ so that $\|f\|_{\al,\xi}<\infty$. 

Our additional requirements from $T_\om,\phi_\om$ and $u_\om$ are specified in the following
\begin{assumption}\label{bound ass}
(i) The random variables $n_\om,D_\om,L_\om$
are bounded and $\gam_\om-1$ is bounded from below by some positive
constant. 

(ii) For any measurable function $g:\cE\to\bbC$  so that $\|g_\om\|_{\al,\xi}\leq 1$ the map  $(\om,x)\to\cL_0 g(\om,x)$ is measurable.

(iii) For $P$-a.a. $\om$ we have $\phi_\om,u_\om\in \cH_\om^{\al,\xi}$  and
the random variables $\|\phi_\om\|_{\al,\xi}$ and $\|u_\om\|_{\al,\xi}$ are bounded.
\end{assumption}
Assumption \ref{bound ass} (i) holds true, for instance when $\cE_\om=\cX=[0,1)$ (and then $L_\om$ does not depend on $\om$), and each  $T_\om$ is a piecewise distance expanding map on the unit interval with a bounded in $\om$ number of monotonicity intervals (on which $T_\om$ is onto $[0,1)$). Several multidimensional examples such as the case when $T_\om x=(m_1(\om) x_1,...,m_d(\om)x_d)\,\text{mod } 1$ can be also be given (here $m_i(\om)\geq 2$ are integer valued and are bounded in $\om$). 
 We refer again the readers to Section 2.1 in \cite{MSU} for more examples of this nature. Different type of examples are random topologically mixing subshifts of finite type with bounded number of symbols, see \cite{Kifer-1998} for a description of such systems.
Note that under Assumption  \ref{bound ass} we have that $\cL_z^{\om,n}(\cH_\om^{\al,\xi})\subset\cH_{\te^n\om}^{\al,\xi}$
and the corresponding operator norm satisfies 
$\|\cL_{it}^{\om,n}\|_{\al,\xi}\leq B(1+|t|)$ where $B$ is some constant (see Lemma  5.6.1 in \cite{book}).

When the random variable $\|\phi_\om\|_{\al,\xi}$  is bounded then Assumption \ref{bound ass} (ii) 
holds true if for any set $A\subset\cX$ with a sufficiently small diameter the set 
\[
\textbf{A}=\{(\om,x): x\in T_\om(A\cap\cE_\om)\}
\]
is measurable (note that when $\cE_\om=\cX$ then $\textbf{A}=T(\Om\times A)$ and that $T_\om$ is injective on $\cE_\om\cap A$ if the diameter of $A$ is smaller than $\xi$). 
Indeed,  in these circumstances we can approximate (uniformly in $\om$) the function $g_\om\cdot e^{\phi_\om}$, by sums of the form $\sum_{i=1}^m a_i(\om)\bbI_{A_i\cap \cE_\om}$, where $\{A_i\}$ is a partition of $\cX$ into sets with small diameter, $\bbI_{A_i\cap \cE_\om}$ is the indicator set of $A_i\cap\cE_\om$ and $a_i(\om)$ are some random variables whose absolute value is bounded by $1$. Then $\cL_0$ can be approximated by
\[
\sum_{i=1}^m a_i(\om)\bbI(x\in T_\om(A_i\cap\cE_\om))
\]
which are measureable functions of $(\om,x)$. Observe that it is enough to assume that the $\textbf{A}$'s are measurbale for sets $A$ which are generated by finite collections of open balls using only the basic Boolean operations. For such sets  $\textbf{A}$ is measurable if $\Om$ is a topological space, $\cF$ is the Borel $\sig$-algebra, $\cE_\om=\cX$ and $T$ is continuous. Another example is the case when $T_\om$ takes only a countable number of ``values" $\{T_1,T_2,...\}$. We refer the readers to Chapter 4 in \cite{MSU} for more aspects of such measurability issues.

\subsubsection{Asymptotic moments: verification of Assumption \ref{BaseAss}}
We recall first the following random complex RPF theorem, which was stated in \cite{book} as Corollary 5.4.2.

\begin{theorem}\label{RPF rand T.O. general}
Suppose that Assumptions \ref{TopExAssRand}, \ref{Ass pair prop} and \ref{bound ass} hold true. Then there exists $\rho>0$ so that $P$-a.s. 
for any $z\in U:=\{\zeta\in\bbC: |\zeta|<\rho\}$ there is a unique triplet consisting of a
 nonzero complex number $\boldsymbol{\la}_\om(z)$ 
a function $\textbf{h}_\om(z)\in\cH^{\al,\xi}_\om$ and a continuous linear functional $\boldsymbol{\nu}_\om(z)$ 
on $\cH^{\al,\xi}_\om$ so that $\boldsymbol{\nu}_\om(z)\textbf{1}=1$, 
\begin{equation}\label{RPF1}
\cL_z^\om \textbf{h}_\om(z)=\boldsymbol{\la}_\om(z)\textbf{h}_{\te\om}(z),\,\,
(\cL_z^\om)^*\boldsymbol{\nu}_{\te\om}(z)=\boldsymbol{\la}_\om(z)\boldsymbol{\nu}_\om(z)\,\text{ and }\,
\boldsymbol{\nu}_\om(z)\big(\textbf{h}_\om(z)\big)=1.
\end{equation}
For any $z\in U$ the maps $\om\to\boldsymbol{\la}(z)$ and $(\om,x)\to \textbf{h}_\om(z)(x),\,(\om,x)\in\cE$
are measurable and the family $\boldsymbol{\nu}_\om(z)$ is measurable in $\om$. 
When $z=t\in\bbR$ then $\boldsymbol{\la}_\om(t)>0$, the function $\textbf{h}_\om(t)$ is strictly
positive, $\boldsymbol{\nu}_\om(t)$ is a probability measure and the equality 
$\boldsymbol{\nu}_{\te\om}(t)\big(\cL_t^\om g)=\boldsymbol{\la}_\om(t)\boldsymbol\nu_{\om}(t)(g)$ holds true for any 
bounded Borel function $g:\cE_\om\to\bbC$ and when the maps $T_\om$ satisfy the assumptions from 
\cite{MSU} then the triplet $(\boldsymbol{\la}_\om(t),\textbf{h}_\om(t),\boldsymbol{\nu}_\om(t))$ coincides with the one 
constructed there with the function $\phi+tu$.

Moreover, this triplet is analytic and uniformly bounded around $0$.
Namely, the maps 
\[
\boldsymbol\la_\om(\cdot):U\to\bbC, \textbf{h}_\om(\cdot):U\to \cH_\om^{\al,\xi}\,\text{ and }
\boldsymbol\nu_\om(\cdot):U\to \big(\cH_\om^{\al,\xi}\big)^*
\]
are analytic, where $(\cH_\om^{\al,\xi})^*$ is the dual space of $\cH_\om^{\al,\xi}$,
and there is a constant $C_0>0$ so that
\begin{equation}\label{UnifBound}
\max\Big(\sup_{z\in U}|\boldsymbol\la_\om(z)|,\, 
\sup_{z\in U}\|\textbf{h}_\om(z)\|_{\al,\xi},\, \sup_{z\in U}
\|\boldsymbol\nu_\om(z)\|_{\al,\xi}\Big)\leq C_0,\,\,P\text{-a.s.}
\end{equation}
where  $\|\nu\|_{\al,\xi}$ is the 
operator norm of a linear functional $\nu:\cH_\om^{\al,\xi}\to\bbC$.

Furthermore, there exists constants $C>0$ and $\del\in(0,1)$ so that $P$-a.s. for any $n\geq1$ and $g\in\cH^{\al,\xi}_\om$ we have
\begin{equation}\label{ExpConv}
\left\|(\boldsymbol\la_{\om,n}(z))^{-1}\cL_z^{\om,n}g-\boldsymbol\nu_\om(z)(g)\textbf{h}_{\te^n\om}(z)\right\|_{\al,\xi}\leq C\|g\|_{\al,\xi}\del^n.
\end{equation}
\end{theorem}
Note that in \cite{book} we have assumed that $\cL_0g$ is measurable  for any measurable function $g$, which was only needed in order to insure that the RPF triplets are measurable in $\om$, but taking a careful look at the arguments in Chapter 4 of \cite{book} we see that only Assumption \ref{bound ass} (ii) is needed.

We remark that the random  (Gibbs) measure $\mu_\om=\textbf{h}_\om(0)\boldsymbol{\nu}_\om(0)$ satisfies 
$(T_\om)_*\mu_\om=\mu_{\te\om}$, or, equivalently,
$\mu=\int_{\Om}\mu_\om dP(\om)$ is $T$-invariant, where $T(\om,x)=(\te\om,T_\om x)$ is the skew product map.
Consider now the transfer operator
operator $\cA_z^\om$ generated by the map $T_\om$ and the function $\tilde\phi_\om+zu_\om$, where $\tilde\phi_\om=\phi_\om+\ln \textbf{h}_\om(0)-\ln \textbf{h}_{\te\om}(0)\circ T_\om-\ln\boldsymbol{\la}_{\om}(0)$. Then $\cA_z^\om\textbf{1}=\textbf{1}$,
\[
\mu_\om(e^{it S_n^\om u})=\mu_{\te^n\om}(\cA_{it}^{\om,n}\textbf{1}),
\]
 $\cA_z^{\om,n}(\cH_\om^{\al,\xi})\subset\cH_{\te^n\om}^{\al,\xi}$
and $\cA_z^\om$ also satisfies the above random complex RPF theorem. More precisely,
Assumption \ref{A3} is satisfied with the norm $\|\cdot\|_1=\|\cdot\|_{\al,\xi}$ and triplets 
$(\la_\om(z),h_\om(z),\nu_\om(z))$ which also satisfy $\la_\om(0)=1$, $h_\om(0)\equiv\textbf{1}$ and $\nu_\om(0)=\mu_\om$.   We refer the readers to the begging of Section 4 in \cite{SeqRPF} for the precise details.
Since $\|\cA_{it}^{\om,n}\textbf{1}\|_{\al,\xi} \leq C\|\cL_{it}^{\om,n}\textbf{1}\|_{\al,\xi}$ for some constant $C>0$ (see \cite{book}, Ch. 5), it is sufficient to show that $\cL_{z}^\om$ satisfy Assumption \ref{A3} with the above norm. 
Henceforth, we will always assume that $\mu_\om(u_\om)=0$ for $P$-a.a. $\om$, which, due to $T$-invariance of $\mu$, implies that $\mu_\om(S_n^\om u)=0$ for each $n\geq1$. This is not really a restriction since we can always replace $u_\om$ by $u_\om-\int u_\om(x)d\mu(x)$.

\subsubsection{Edgeworth expansions of order one: verification of Assumption \ref{A3}}

The first part of Assumption \ref{A3} is satisfied due to the random Lasota-Yorke type inequality stated in 
Lemma 5.6.1 in \cite{book}.
Next, under the following assumption we showed in Chapter 7 of \cite{book} that the rest of Assumption \ref{A3} holds true for the operators $\cL_z^\om$.

\begin{assumption}\label{PerPointAss}
(i)  $\Om$ is a topological space, $\cF$ is the corresponding 
Borel $\sig$-algebra, $(\Om,\cF,P,\te)$ is ergodic and $\te$ has a periodic point, namely there exist $\om_0\in\Om$ and $n_0\in\bbN$ so that $\te^{n_0}\om_0=\om_0$. Moreover, $P(U)>0$ for any open set $U$ which intersects the finite orbit $\{\te^i\om_0:\,0\leq i<n_0\}$ and
the spaces $\cE_\om$ are locally independent of $\om$ around the points $\te^{j}\om_0,\, 0\leq j<n_0$.

(ii) The map $\te$ is continuous at the points $\om_0,\te\om_0,...,\te^{n_0-1}\om_0$, where $\om_0$ and $n_0$ come from (i), and for any compact set $J\subset\bbR\setminus\{0\}$ the family of maps 
$\om\to\cL_{it}^\om,\,t\in J$ is equicontinuous continuous (with respect to the operator norm $\|\cdot\|_{\al,\xi}$) at the points $\te^{j}\om_0,\,0\leq j<n_0$.

(iii) The function $S_{n_0}^{\om_0}u$ is non-arithmetic 
(aperiodic) with respect to the map $T^{n_0}_{\om_0}$ in the classical sense of \cite{HH}, namely 
if for any $t\in\bbR\setminus\{0\}$ there exist no 
nonzero $g\in\cH_{\om_0}^{\al,\xi}$ and $\la\in\bbC$, $|\la|=1$ 
such that 
\begin{equation}\label{lattice condition 1.}
e^{itS_{n_0}^{\om_0}u}g=\la g\circ T_{\om_0}^{n_0}.
\end{equation}
\end{assumption}
Note that under Assumption \ref{PerPointAss} we showed in Chapter 7 of \cite{book} that, in fact, 
 for any compact set $J\subset\bbR\setminus\{0\}$, $P$-.a.s for any $n\geq1$ we have 
\begin{equation}\label{ExpR}
\sup_{t\in J}\|\cA^{\om,n}_{it}\|_1\leq d_\om 2^{-cn}
\end{equation}
for some random variable $d_\om>0$ and a constant $c>0$.
In Chapter 7 of \cite{book} we stated that (\ref{ExpR}) holds true under the additional assumption that $P(U)>0$ for any open set $U$, but in the proof we relied on that only for $U$'s which intersect the orbit $\{\om_0,\te\om_0,...,\te^{n_0-1}\om_0\}$.
The condition about  continuity of $\om\to\cL_{it}^\om$ holds true, for instance, 
when the maps $\om\to \phi_\om, u_\om\in\cH^{\al,\xi}_\om$ are continuous at the 
points $\te^{j}\om_0,\, 0\leq j<n_0$ and the  branches of $T_\om$ are either locally constant as functions of $\om$ there, or they are continuous 
at these points with respect to an appropriate topology. For instance, we can just assume that $T_\om$ does not depend on $\om$ around the points $\te^i\om_0$. This is the case when $(\Om,\cF,\sig)$  is a shift system built over a countable alphabet $\mathscr A$ so that $T_\om=T_{\om_0}$ depends only on the $0$-th coordinate of $\om=(\om_i)$. We can consider any sub-shift of $\Om$ instead, assuming it contains a point of the form $(...a,a,a,...)$  where $a\in\mathscr A^{n_0}$ for some $n_0$. Of course, we we also have to assume that $P(U)>0$ for the appropriate $U$'s. 
When considering interval maps $T_\om$ with deterministic monotonicity intervals $I_1,...,I_D$ (or similar multidimensional examples), then the condition about the continuity holds true if the restriction of the inverse of each map $T_\om|_{I_i}$ is continuous in $\om$ at $\om=\te^j\om_0$, $i=0,...,n_0-1$. 

Note that when $\cX$ is a smooth connected Riemannian manifold then we can also consider the norms $\|g\|_{C^1}=\sup|g|+\sup|Dg|$ instead of $\|\cdot\|_{\al,\xi}$. In this case when each $T_\om$ is also smooth and the map $\om\to T_\om$ is continuous in the $C^1$ topology then the transfer operators $\om\to\cL_\om$ are continuous in the operator-norm topology (corresponding to the norm $\|\cdot\|_{C^1}$), see Proposition 5.3 in \cite{castro}. Therefore, when considering smooth functions $\phi_\om$ and $u_\om$ which are continuous in the $C^1$-topology as functions of $\om$ we get that the conditions in item (ii) above are satisfied with $\|\cdot\|_{C^1}$ instead of $\|\cdot\|_{\al,\xi}$, and all our results hold true.


\subsubsection{Edgeworth expansions of high orders: verification of Assumption \ref{A4}}
Now we will explain in which circumstances Assumption \ref{A4} is satisfied with any  $r$. We first recall the following results from  \cite{[3]}. Let $\mathscr{L}_b,\,b\in\bbR$ be the deteministic transfer operators  which are generated by two  piecewise smooth expanding functions $f:\bbT\to\bbT$ and $\tau:\bbT\to\bbT$ and is given by the formula
\[
\mathscr{L}_bh(x)=\sum_{y\in f^{-1}\{x\}}|f'(y)|^{-1}h(y)e^{ib\tau(y)}.
\]
Let $\|\cdot\|_2=\|\cdot\|_{{\textbf{BV}}}=\|\cdot\|_{L^1(\bbT)}+\text{Var}(\cdot)$ be the norm in the space of functions with bounded variation. 
When $\tau$ is not cohomologous to a piecewise constant function (with respect to the base map $f$, see \cite{[3]}), 
then by Proposition 1 in \cite{[3]}, there exist positive constants $b_0,\rho$ and $\gam_2$ such that
\begin{equation}\label{Prop1}
\|\mathscr{L}_b^{n(b)}\|_{(b)}\leq e^{-n(b)\gam_2}\,\,\text{ for any }\,\,|b|\geq b_0,\,n(b):=[\rho\ln|b|]
\end{equation}
where $\|h\|_{(b)}=(1+|b|)^{-1}\|h\|_{\textbf{BV}}+\int_{\bbT}|h(x)|dx$. In \cite{Liverani} the authors showed that this implies 
that Assumption \ref{A4} holds true in this deterministic case.  

Next, we assume that $\cE_\om=\bbT=[0,1)$ for any $\om$ and that, in addition to the requirements at the beginning of Section \ref{SkewProdSec}, the functions $T_\om$ and $u_\om$ are piecewise smooth expanding functions and that $u_\om$ takes values in $\bbT$. Then Assumptions \ref{BaseAss} holds true. 
In order to keep our notations in line with the ones in \cite{[3]}, we set  $f_\om=T_\om$ and $\tau_\om=u_\om$. We also set
$\phi_\om=-\ln |f_\om'|$. 
We assume here that there exists no random family $\eta_\om$ of functions $\eta_\om:\bbT\to\bbT$ so that $\tau_\om-\eta_{\te\om}\circ f_\om+\eta_\om$ is piecewise constant ($P$-a.s.).
Then, when $(\Om,\cF,P,\te)$ is ergodic, an appropriate quenched version of Proposition  1 in \cite{[3]}
 holds true when we replace $f^n$ and $\sum_{j=0}^{n-1}\tau\circ f^j$ with $f_\om^n:=f_{\te^{n-1}\om}\circ f_{\te^{n-2}\om}\circ\cdots\circ f_\om$ and $\sum_{j=0}^{n-1}\tau_{\te^j\om}\circ f_\om^j$, respectively, when all the conditions in \cite{[3]} hold true uniformly in $\om$ 
 in the following sense:
 there exist random variables $\rho_\om,\gam_2(\om),b_\om>0$ so that $P$-a.s., for any $b\in\bbR$,
\begin{equation}\label{RandomProp1}
\|\cL^{\om,n_\om(b)}_{ib}\|_{(b)}\leq e^{-n_\om(b)\gam_2(\om)}\,\,\text{ for any }\,\,|b|\geq b_\om,\,n_\om(b):=[\rho_\om\ln|b|].
\end{equation}

Before explaining how the proof of (\ref{RandomProp1}) differs from the proof of (\ref{Prop1}), we will rely on (\ref{RandomProp1}) in order to show that Assumption \ref{A4} holds true. Indeed, let 
$b_0,\rho_1,\rho_2,\gam_2>0$  be so that 
\[
P\{\om:\,\rho_1\leq \rho_\om\leq\rho_2,\,\gam_2(\om)\geq\gam_2\,\text{ and }
b_\om\leq b_0\}:=a>0.
\]
Then by ergodicity of 
$(\Om,\cF,P,\te)$, for $P$-almost any $\om$ there exists a sequence of positive integers $n_1<n_2<n_3<...$ so that for each $i$, 
\[
\rho_1\leq \rho_{\te^{n_i}\om}\leq\rho_2,\,\, b_{\te^{n_i}\om}\leq b_0\,\text{ and }\, \gam_2(\te^{n_i}\om)\geq \gam_2
\]
and $\lim_{k\to\infty}k^{-1}n_k=a^{-1}$. It follows that for any  $c>0$ there exists a constant $c_1>0$ so that 
for any sufficiently large $n$, there are indexes $i_1<i_2<...<i_s$, $s=s_n=[c_1\frac{n}{\ln n}]$, $i_j=i_j(n)$ with the property that $n_{i_j}+c\ln n<n_{i_{j+1}}<n-c\ln n$ for any $j=1,2,...,s-1$. In fact, since $n^{-1}\{\max k: n_k\leq n\}$ converges as $n\to\infty$ to $a$, we can always take $i_1=1$, $i_s<\frac n2-c\ln n$ and $c_1=\frac{a}{4(c+1)}$.

Next, it follows exactly as in Lemma 2 in \cite{[3]} that $\|\cL_{ib}^{\om,n}\|_{(b)}\leq C$ for any $b$ and $n$, where $C\geq1$ is some constant (in fact, it is possible to take $C=1$, but it does not matter for our purpose). Let $d>0$
and let $b_1>\max(b_0,1)$ be so that $b_1^{\frac12\rho_1\gam_2}>2C$. 
 Then for any real $b$ so that $b_1\leq |b|\leq n^{d}$ we have 
$\rho_2\ln|b|\leq\rho_2d\ln:=c\ln n$. Let $c_1$ and $i_1<i_2<...<i_s$ (for a sufficiently large $n$) be as describe in the previous paragraph  with $c=\rho_2d$.
For $n$ large enough, set 
\begin{eqnarray*}
\om_j=\om_j(n)=\te^{n_{i_j}}\om,\, m_j(b)=[\rho_{\om_j}\ln|b|],\, 
d_j(b)=d_j(b,n)=n_{i_{j+1}}-n_{i_{j}}-m_j(b),\\ \text{ where } n_{s+1}:=n, \text{ and } 
\om'_{j}(b)=\om'_j(b,n)=\te^{[\rho_{\om_{j}}\ln|b|]}\om_j.
\end{eqnarray*}
By writing 
\[
\cL_{ib}^{\om,n}=\cL_{it}^{\om'_{s}(b),d_s(b)}\circ\cL_{ib}^{\om_s,m_s(b)}\circ\cdots\circ\cL_{ib}^{\om_2,m_2(b)}\circ\cL^{\om'_{1}(b),d_j(b)}_{ib}\circ\cL_{ib}^{\om_1,m_1(b)}\circ\cL_{ib}^{\om,n_{i_1}}
\]
and using submultiplicativity of norms of operators, we deduce that $P$-a.s.
for any sufficiently large $n$ and real $b$ so that $b_1\leq |b|\leq n^{d}$ we have
\[
\|\cL_{ib}^{\om,n}\|_{(b)}\leq C^{1+\frac{c_1 n}{\ln n}}e^{-\frac12c_1\rho_1\gam_2\ln|b|\frac{n}{\ln n}}
=C\big(|b|^{\frac12\rho_1\gam_2}C^{-1}\big)^{-\frac{c_1 n}{\ln n}}\leq 
C2^{-\frac{c_1 n}{\ln n}}
\]
which implies that (\ref{EqA4}) from Assumption \ref{A4} holds true with $d_1=d$  and $K_\om=b_1$ (since $|b|\leq n^d\leq 2^{2d\ln n}$).

Now we will explain the differences between the derivations of (\ref{Prop1}) and (\ref{RandomProp1}), 
and show that a random variable $R_\om$ with the properties required in Assumption \ref{A4} exists.
 First, in the above circumstances by Theorem 2.2 in \cite{Kifer-Thermo} we have $\boldsymbol{\la}_\om(0)=1$
and $\boldsymbol{\nu}_\om(0)=\text{Lebesgue}$.
By taking $\al=1$, the arguments in Chapter 5 of \cite{book} show that 
$\|h_\om(0)\|_{\textbf{BV}}$ and $\|1/h_\om(0)\|_{\textbf{BV}}$ are both bounded random variables. Therefore, in order to show that Assumption \ref{A4} hold true with the operators $\cA_{it}^{\om,n}$, it is sufficient to show that it holds true with the operators $\cL_{it}^{\om,n}$ (and the norm $\|\cdot\|_2=\|\cdot\|_{\textbf{BV}}$).
Let
$\Lambda$ and $\tilde\la>1$ be positive constants so that $P$-a.s.
\[
\Lambda\geq\sup|f_\om'|\geq\inf|f_\om'|\geq\tilde\la.
\]
Then all the arguments in the proofs  in Section 2 of \cite{[3]} proceed similarly with the transfer operators $\cL_{ib}^{\om,n}$ in place of $\mathscr L_{ib}^{n}$, since they only rely on the expansion properties of $f$ and $\tau$. In particular, $\|\cL_{it}^{\om,n}\|_{\textbf{BV}}\leq R(1+|t|)$ for some constant $R>0$. Moreover, it is possible to construct  partitions $\Om_n=\Om_n^\om,\,n\geq0$ almost exactly as in \cite{[3]} (see the paragraph preceding  Lemma 4 in \cite{[3]}), which in our case will  also depend on $\om$: their recursive  
construction in our nondeterministic setup proceeds exactly as in \cite{[3]}, expect that in the $n$-th step of the construction we consider the function $f_{\te^{n}\om}$ in place of $f$. The estimates derived in Lemma 4 in \cite{[3]} from there  hold true also for the $\Om_n^\om$'s, and the proof goes exactly in the same way as in \cite{[3]}.

Now we will explain how to generalize the ideas from Section 3 in \cite{[3]} to our non-deterministic setting. 
Set $F_\om(x,u)=(f_\om(x),u+\tau_\om(x))$, where $x,y\in\bbT$.
Then 
\[
F^n_\om(x,u):=F_{\te^{n-1}\om}\circ F_{\te^{n-2}\om}\circ\cdots\circ
F_\om(x,u)=\big(f_\om^n(x),u+\sum_{j=0}^{n-1}\tau_{\te^j\om}\circ f_\om^j(x)\big)
\]
and the differential $DF_\om(x,u)$ of $F_\om$ does not depend on $u$. 
 The fact that $f_\om$ and $\tau_\om$ are expanding uniformly in $\om$ implies the  cones $\cK=\{(\al,\be): |\beta||\al|^{-1}\leq C_1\}$ are invariant under each $DF_\om$. Here $C_1$ is a constant so that $P$-a.s.,
\[
C_1\geq\sup|\tau_\om'|(\tilde\la-1)^{-1}.
\]
Thus, the notation of transversality from Section 3 in \cite{[3]} is naturally extended to our setup, using $DF^n_\om(x)$ in place of the differential of $F^n(x,u)=\big(f^n(x),u+\sum_{j=0}^{n-1}\tau\circ f^j(x)\big)$. 
That is, for any $y\in\bbT$, we will say that $x_1,x_2\in (f_\om^n)^{-1}\{y\}$ are transversal if $DF^n_\om(x_1)\cK\bigcap DF^n_\om(x_2)\cK=\{0\}$.
Set $J_{\om,n}=|(f_\om^n)'|^{-1}$, and
define the quantity $\varphi_\om(n)$ by
\[
\varphi_\om(n)=\sup_{y\in\bbT}\sup_{x_1\in (f_\om^n)^{-1}\{y\}}
\sum_{x_2\in A_\om(x_1,n,y)}J_{\om,n}(x_2)
\]
where $A_\om(x_1,n,y)$ is the set of all preimages of $y$ by $f_\om^n$ which are not transversal to $x_1$. We also define
the quantity $\tilde\varphi_\om(n)$ 
by 
\[
\tilde\varphi_\om(n)=\sup_{y}\sup_{L}\sum_{\substack{
x\in (f_\om^n)^{-1}\{y\}\\L\subset DF_\om^n(x)\cK}}J_{\om,n}(x)\frac{\textbf{h}_\om(x)}{\textbf{h}_{\te^n\om}(y)}
\]
where $L\subset\bbR^2$ is a line which passes through the origin
and $\textbf{h}_\om=\textbf{h}_\om(0)$ is the  random function from the RPF triplet, i.e. the one satsifying that $\cL_0^\om \textbf{h}_\om=\textbf{h}_{\te\om}$ (recall that $\boldsymbol{\la}_\om(0)=1$). Note that (see Chapter 5 in \cite{book}) the function $\textbf{h}_\om$ is strictly positive and is bounded and bounded away from $0$ uniformly in $\om$.
 Then $\tilde\varphi_\om(n)\leq 1$ and it satisfies that 
\[
\tilde\varphi_\om(n+m)\leq \tilde\varphi_\om(n)\cdot \tilde\varphi_{\te^n\om}(m).
\]
By Kingman's subadditive ergodic theorem, the limit $\lim_{n\to\infty}(\tilde\varphi_\om(n))^{\frac1n}$ exists, $P$-a.s., and it does not depend on $\om$. When this limit equals $1$ then, by considering the subadditive sequence $g_n=\int \ln\tilde\varphi_\om(n)dP(\om)$, it follows that $\tilde\varphi_\om(n)=1$ for $P$-a.a. $\om$ and all $n$'s.
Relying on that, the arguments from the proof of Lemma 7 in \cite{[3]} show that
\[
\limsup_{n\to\infty}\varphi_\om(n)<1\,\,\text{for }\,P\text{-a.a. } \om
\]
if there exists no random family $\eta_\om$ of functions $\eta_\om:\bbT\to\bbT$ so that $\tau_\om-\eta_{\te\om}\circ f_\om+\eta_\om$ is piecewise constant ($P$-a.s.). The rest of the derivation of (\ref{RandomProp1}) proceeds now as in Section 4 of  \cite{[3]}, relying on the uniform (in $\om$) expansion properties of the functions $f_\om$ and $\tau_\om$.

\begin{remark}
The assumption that $u_\om$ is continuous can be dropped since all the results from \cite{book} hold true also when $u_\om$ is piecewise H\"older continuous. We will not present here all the details and instead refer the readers to Assumption 5.2.1 in  \cite{book}, and to the results proceeding it.
\end{remark}

\subsubsection{Convergence rates towards the asymptotic moments}
In the circumstances of Section \ref{MomRateSec}, 
the operators $\cA_z^{\om}$ do not depend only on the $0$-th coordinate even when $\cL_z^\om$ depends on it (only). Still, in order to obtain the convergence rates stated in Section \ref{MixSec}, it is enough to obtain the appropriate almost sure rates in the convergence of the ergodic averages $n^{-1}\sum_{j=0}\Pi_{\te^j\om}^{(k)}(0)$ (see the arguments in Section \ref{MomSec}. Let us  denote by $\boldsymbol{\Pi}_\om(z)$ the pressure function corresponding to $\boldsymbol\la_{\om}(z)$. In (4.2) in \cite{SeqRPF} we have proved that there exist constants $c_1$ and $\rho_1$ so that  $P$-a.s. for any $z\in\bbC$ with $|z|<\rho_1$ and any $n\geq 1$ we have
\[
\left|\sum_{j=0}^{n-1}\Pi_{\te^j\om}(z)-\sum_{j=0}^{n-1}\boldsymbol\Pi_{\te^j\om}(z)\right|\leq c_1.
\]
Note that it is indeed possible to bound from below $\min |h_\om(z)|$ when $z$ belongs to some deterministic neighborhood of the origin since $h_\om(x)\geq C>0$ for some constant $C>0$ which does not depend on $\om$ and $x$ (see (5.9.4) and (4.3.32) in \cite{book}).
Therefore, when $\cL_z^\om$ depend only on the $0$-th coordinate, we can just repeat the argument from Section \ref{MomRateSec}  with $\boldsymbol\Pi$ in place of $\Pi$ we obtain all the results stated in Section \ref{MixSec} under the mixing conditions from Section \ref{MomRateSec}. 


\subsection{Markov chains with transition densities}\label{Densities}
For any $\om\in\Om$ denote by $B_\om=H_1^\om$ the Banach space of all bounded
Borel functions $g:\cE_\om\to\bbC$ together with the supremum norm $\|\cdot\|_\infty$.
For any $g:\cE\to\bbC$ consider the functions $g_\om:\cE_\om\to\bbC$ given by 
$g_\om(x)=g(\om,x)$.
Then by Lemma 5.1.3 in \cite{book}, the norm $\om\to\|g_\om\|_\infty$ is a $\cF$-measurable
function of $\om$, for any measurable $g:\cE\to\bbC$.

Let $r_\om=r_\om(x,y):\cE_\om\times\cE_{\te\om}\to[0,\infty),\,\om\in\Om$ be a family of 
integrable in $y$ Borel measurable functions, $m_\om,\,\om\in\Om$ be a family of Borel probability measures
on $\cE_\om$ and $u:\cE\to\bbR$ be a measurable function so that $u_\om\in B_\om,\,$ $P$-a.s.   
and that the random variable $\sup|u_\om|=\|u_\om\|_\infty$ is bounded.
Consider the family of random operators $R_z^\om,\,z\in\bbC$ which map (bounded) 
Borel functions $g$ on $\cE_{\te\om}$ to Borel measurable functions on $\cE_\om$ 
by the formula
\begin{equation}\label{IntOp}
R^\om_zg(x)=\int_{\cE_{\te\om}} r_\om(x,y)e^{zu_{\te\om}(y)}g(y)dm_{\te\om}(y).
\end{equation} 
We will assume that 
$R_0^\om$ are Markov operators, namely that $R_0^\om \textbf{1}=\textbf{1}$
where $\textbf{1}$ is the function which takes the constant value $1$ on $\cE_{\te\om}$.
Observe that 
\[
\|R_0^\om\|_\infty:=
\sup_{g\in B_{\te\om}:\|g\|_\infty\leq 1}\|R_0^\om g\|_\infty=\|R_0^\om\textbf{1}\|_\infty
\]
and therefore for $P$-a.a. $\om$ 
we have $\|R_z^\om\|_\infty<\infty$ for any 
$z\in\bbC$, namely,  $R_z^\om$ is a continuous linear operator between the Banach spaces 
$B_{\te\om}$ and $B_\om$.

\begin{assumption}\label{Meas ass}
The functions $r_\om$ are measurable in $\om$ and the 
the measures $m_\om$ are measurable in the sense that
the map $\om\to\int_{\cE_{\om}} g_{\om}(x)dm_{\te\om}(x)$ is measurable for any 
measurable function $g:\cE\to\bbC$ so that $g_\om(\cdot)$ is bounded $P$-a.s.
\end{assumption}
Under this assumption the map $(\om,x)\to R_0^\om g_{\te\om}(x)$ is measurable for any measurable function $g=g(\om,x)=g_\om(x)$.

For any $\om\in\Om$, $n\in\bbN$ and $z\in\bbC$ consider the
$n$-th order iterates $R_z^{\om,n}:B_{\te^n\om}\to B_\om$
given by 
\begin{equation}\label{Int Op oter 1}
R_z^{\om,n}=R_z^\om\circ R_z^{\te\om}\circ\cdots\circ R_z^{\te^{n-1}\om}.
\end{equation}
Then we can write
\[
R_0^{\om,n}g(x)=\int _{\cE_{\te^n\om}}r_\om(n,x,y)g(y)dm_{\te^n\om}(y)
\]
for some family $r_\om(n,\cdot,\cdot)=r_\om(n,x,y):\cE_\om\times\cE_{\te^n\om}\to[0,\infty)$
of integrable in $y$ Borel
measurable functions.
We will assume that the following random version of the two sided Doeblin condition holds true. 
\begin{assumption}\label{Doeb}
There exist a bounded random variable $j_\om\in\bbN$ 
and random variables $0<\al_m(\om)\leq 1$, $m\in\bbN$ such that $P$-a.s.,
\begin{equation}\label{al j om}
\al_m(\om)\leq r_\om(m,x,y)\leq\big(\al_m(\om)\big)^{-1},
\end{equation}
for any $m\geq j_\om\,$, $x\in\cE_\om$ and $y\in\cE_{\te^m\om}$.
Moreover, let $j_0$ be so that $j_\om\leq j_0$, $P$-a.s. Then there exists $\al>0$
so that $\al_{j_0}(\om)\geq\al$.
\end{assumption}
Under the above assumptions, we showed in Chapter 6 of \cite{book} that the family of operators $\cA_z^\om=R_z^{\te^{-1}\om}$ satisfies Assumption \ref{BaseAss} with the measure preserving system $(\Om,\cF,P,\te^{-1})$.

 Let $\mu_\om$ be any probability measure on $\cE_\om$ and let $\xi_n^{\te^n\om},\,n\geq1$
be the Markov chain with initial distribution $\mu_\om$ whose $n$-th step operator is given by 
$R_0^{\om,n}$. Set 
\[
S_n^\om=\sum_{j=0}^{n-1}u_{\te^j\om}(\xi_j^{\te^j\om}).
\]
Then, 
\[
\bbE e^{itS_n^\om}=\int\mu_\om(R_{it}^{\om,n}\textbf{1})dP(\om)
\]
and therefore all the results stated in Section \ref{sec2} hold true with the random variables $S_n^{\te^{-n}\om}$.


\subsubsection{Edgeworth expansions of order one: verification of Assumptions \ref{A3}}
The first part of Assumption \ref{A3} holds true with the supremum norm since $\|R_{it}^\om\|_\infty\leq\|R_0^\om\|_\infty=1$. 
Let $\om_0\in\Om$ and $n_0\in\bbN$ be so that $\te^{n_0}\om_0=\om_0$. Suppose that for any $t\in\bbR\setminus\{0\}$ the spectral radius $r(t)$ of the operator $R_{it}^{\om_0,n_0}$ is strictly less than $1$. 
We refer the reader to Section 6.3 in \cite{Liverani} for conditions guaranteeing that $r(t)<1, t\not=0$ (see the verification of Assumption (A3) there). Then by Lemma 2.10.4 in  \cite{book}
 the second part Assumption \ref{A3} holds true when Assumption \ref{PerPointAss} is satisfied with $R_z^\om$ in place of $\cL_z^\om$ and with the norm $\|\cdot\|_\infty$ instead of $\|\cdot\|_{\al,\xi}$.  
 Note that in Chapter 7 we have also verified Assumption \ref{A3} under certain type of non-lattice condition. These results were obtained under the assumption that $P(j_\om=j_{\te\om}=1)>0$ (see Assumption 7.1.4 there), and it is also possible to extend them to more general cases (see Section 7.4.4 there).

\subsubsection{Edgeworth expansions of high orders: verification of Assumptions \ref{A4}}
We will show that Assumption \ref{A4} holds true with $H_2^\om=B_\om$  under the following 
\begin{assumption}\label{AboveAss}
The system $(\Om,\cF,P,\te)$ is ergodic and
the probability that
\begin{equation}\label{Cond}
\limsup_{|t|\to\infty}\|R_{it}^{\om,2}\|_\infty<1
\end{equation}
is positive.
\end{assumption}
Before verifying Assumption \ref{A4} we will provide conditions guaranteeing it holds true.
The following result is proved exactly as in \cite{Liverani}.
\begin{lemma}\label{Lem}
Let $\om\in\Om$ and suppose that $\cE_\om=\cE_{\te\om}$ is a compact connected smooth
manifold, that $u_{\te\om}$ and $u_{\te^2\om}$ are piecewise smooth and that $r_\om$ is a $C^1$-function. Moreover, assume that
$j_\om=j_{\te\om}=1$ and that the function $(y_1,y_2)\to u_{\te\om}(y_1)+u_{\te^2\om}(y_2)$ is not piecewise
constant. Then condition (\ref{Cond}) holds true with the above $\om$.
\end{lemma}
Assumption \ref{Cond} holds true  when
the probability that $\om$ satisfies the conditions of Lemma \ref{Lem} is positive.

Now we will verify Assumption \ref{A4}.
Under Assumption \ref{AboveAss}, there exist $\ve_0,r_0>0$ so that $p_0=P(\Gam_0)>0$, where $\Gam_0$ is the set of all $\om$'s so that 
\begin{equation*}
\sup_{t: |t|\geq r_0}\|R_{it}^{\om,2}\|_\infty\leq 1-\ve_0.
\end{equation*}
Since $\te$ is ergodic $P$-a.s. there exists a sequence $n_1<n_2<n_3<...$ of
positive integers so that
\[
\lim_{k\to\infty}\frac{n_k}{k}=\frac1{p_0}\,\,\text{ and }\te^{-n_i}\om\in\Gam_0\,\,\forall\,i\in\bbN.
\]
Set $m_i=n_{2i}$, $\om_i=\te^{-m_i}\om$ and write
\begin{eqnarray*}
\cA_{it}^{\te\om,n}=R^{\te^{-n+1}\om,n}_{it}=R_{it}^{\te^{-n+1}\om}\circ R_{it}^{\te^{-n+2}\om}\circ\cdots\circ
R_{it}^{\om}\\
=R_{it}^{\te^{2}\om_{k_{n}},n-k_{n}-2}\circ R_{it}^{\om_{k_{n}},2}\circ\cdots\circ R_{it}^{\om_2,2}\circ R_{it}^{\te^{2}\om_1,m_2-m_1-2}R_{it}^{\om_1,2}\circ R_{it}^{\te^{2}\om_2,m_1-2}
\end{eqnarray*}
where $k_n=\max\{i: m_i\leq n-3\}$. Since $\|R_{it}^\om\|_\infty\leq 1$ and 
$k_n/n$ converges to $\frac 12p_0$ as $n\to\infty$, we obtain from submultiplicativity of norms of operators that
$P$-a.s., for any $t\in\bbR$ so that $|t|\geq r_0$ and any sufficiently large $n$,
\[
\|\cA_{it}^{\om,n}\|_\infty\leq (1-\ve_0)^{k_{n-2}}\leq (1-\ve_0)^{\frac14p_0n}
\]
which implies that Assumption \ref{A4} holds true. We note that Assumption \ref{A4} also
holds true with the norms $\|\cdot\|_\infty$ when appropriate random versions of conditions $(C)$ and $(B_k)$ from \cite{Neg2} are satisfied
(with positive probability).


\begin{thebibliography}{Bow75}

\bibliographystyle{alpha}
\itemsep=\smallskipamount





\bibitem{Aimino}
R. Aimino, M. Nicol and S. Vaienti. {\em Annealed and quenched limit theorems for random expanding dynamical systems}, Probab. Th. Rel. Fields 162, 233-274, (2015).





 \bibitem{Bow}
R. Bowen, {\em Equilibrium states and the ergodic theory of Anosov diffeomorphisms},
 Lecture Notes in Mathematics, volume 470, Springer Verlag, 1975.
 
\bibitem{[3]} 
O. Butterley, and E. Peyman {\em Exponential mixing for skew products with discontinuities}, Trans. Amer. Math. Soc. 369 (2017), no. 2, 783-803.
 

 
\bibitem{Br}
R.C. Bradley, {\em Introduction to Strong Mixing Conditions}, Volume 1, Kendrick Press, Heber City, 2007.





\bibitem{Coelho}
Z. Coelho, W. Parry {\em Central limit asymptotics for shifts of finite type}, Israel J. Math.
69, (1990), no. 2, 235-249.



\bibitem{CV}
C. Castaing and M.Valadier, {\em Convex analysis and measurable multifunctions}, 
Lecture Notes Math., vol. 580, Springer, New York, 1977.


\bibitem{drag}
D. Dragi\v{c}evi\'c, G. Froyland, C. Gonz\'alez-Tokman and S. Vaienti,
{\em A spectral approach for quenched limit theorems for random expanding dynamical systems},
Commun. Math. Phys. 360, 1121-1187 (2018).

\bibitem{castro}
A. Castro, P. Varandas. {\em Equilibrium states for non-uniformly expanding maps: decay of
correlations and strong stability}. Annales de l'Institut Henri Poincar\'e - Analyse non Lineaire,
 (2013) 225-249,



\bibitem{Dub1}
L. Dubois, {\em Projective metrics and contraction principles for
complex cones}, J. London Math.
Soc. 79 (2009), 719-737.


\bibitem{Dub2}
L. Dubois, {\em An explicit Berry-Ess\'een bound for uniformly expanding maps on
 the interval}, Israel J. Math. 186 (2011), 221-250.


\bibitem{Feller}
W. Feller, {\em An introduction to probability theory and its applications}, Vol. II., Second edition, John Wiley and Sons, Inc., New York-London-Sydney, 1971. 

\bibitem{GH}
Y. Guivar\'ch and J. Hardy, {\em Th\'eor\`emes limites pour une classe de cha\^ines de Markov et
applications aux diff\'eomorphismes d'Anosov},
 Ann. Inst. H. Poincar\'e Probab. Statist. 24 (1988), no. 1, 73-98.


\bibitem{book}
Y. Hafouta and Yu. Kifer, {\em Nonconventional limit theorems and random dynamics}, 
World Scientific, Singapore, 2018.





\bibitem{Haf-MD}
Y. Hafouta, {\em Nonconventional moderate deviations theorems and exponential concentration inequalities}, accepted to publication in  Ann. Inst. H. Poincar\'e Probab. Statist.


\bibitem{Annealed} Y. Hafouta, {\em Limit theorems for some skew products with mixing base maps}, to appear in Ergod. Theor. Dyn. Syst, 
DOI: https://doi.org/10.1017/etds.2019.48.


\bibitem{SeqRPF}
Y. Hafouta, {\em Limit theorems for some time dependent expanding dynamical systems}, arXiv preprint, 1903.04018.

\bibitem{Haf-Arrays}
Y. Hafouta, A functional CLT for nonconventional polynomial arrays, accepted in publication to Discrete Contin. Dyn. Syst
arXiv preprint, 1907.03303.

\bibitem{Ste}
O.Hella and M. Stenlund, {\em Quenched normal approximation for random sequences of transformations},  J Stat Phys (2019) doi:10.1007/s10955-019-02390-5.
 
\bibitem{HH}
H. Hennion and L. Herv\'e, {\em Limit Theorems for Markov Chains and Stochastic Properties of Dynamical
Systems by Quasi-Compactness}, Lecture Notes in Mathematics vol. 1766, Springer, Berlin, 2001.



\bibitem{Kifer-1996}
Yu. Kifer, {\em Perron-Frobenius theorem, large deviations, and random perturbations
in random environments},
 Math. Z. 222(4) (1996), 677-698.

\bibitem{Kifer-1998}
Yu. Kifer, {\em Limit theorems for random transformations and processes in random
environments},
Trans. Amer. Math. Soc. 350 (1998), 1481-1518.

\bibitem{Kifer-Thermo}
 Yu. Kifer, {\em Thermodynamic formalism for random transformations revisited},
Stoch. Dyn. 8 (2008), 77-102.


\bibitem{Lin}
Lin Z, Bai L, {\em probability inequalities}, Springer, 2010.

\bibitem{Liverani}
K. Fernando and C. Liverani, {\em Edgeworth expansions for weakly dependent random variables}, preprint, arXiv 1803.07667, 2018.


\bibitem{Mil}
V.D. Milman and G. Schechtman, {\em Asymptotic theory of finite-dimensional normed spaces}, Lecture
Notes in Mathematics, Vol. 1200, Berlin: Springer-Verlag, 1986, (With an appendix by M.
Gromov).

\bibitem{MSU}
V. Mayer, B. Skorulski and M. Urba\'nski, {\em Distance expanding random mappings,
thermodynamical formalism, Gibbs measures and fractal geometry},
Lecture Notes in Mathematics, vol. 2036 (2011), Springer.
 

\bibitem{Neg1}
S.V. Nagaev, {\em Some limit theorems for stationary Markov chains},
Theory Probab. Appl. 2 (1957), 378-406.

\bibitem{Neg2}
S.V. Nagaev, {\em More exact statements of limit theorems for homogeneous Markov chains}, 
Theory Probab. Appl. 6 (1961), 62-81.

\bibitem{Rug}
H.H. Rugh, {\em Cones and gauges in complex spaces: Spectral gaps and complex
 Perron-Frobenius theory}, 
Ann. Math. 171 (2010), 1707-1752.






\bibitem{Young1}
L.S. Young, {\em Statistical properties of dynamical systems with some hyperbolicity}, 
 Ann. Math. 7 (1998) 585-650.

\bibitem{Young2}
L.S. Young, {\em  Recurrence time and rate of mixing}, Israel J. Math. 110 (1999) 153-88.










\end{thebibliography}
\end{document}